\documentclass[12pt,reqno]{amsart}
\usepackage{amsfonts}
\usepackage{pdfsync}
\usepackage{amssymb, amsmath}
 \usepackage{graphicx,graphics,xcolor,epsfig,setspace}
\usepackage[section]{placeins}
\usepackage{comment}
\usepackage{marginnote}

\def\sqr#1#2{{\vcenter{\vbox{\hrule height.#2pt
        \hbox{\vrule width.#2pt height#1pt \kern#2pt
        \vrule width.#2pt}
        \hrule height.#2pt}}}}

\numberwithin{equation}{section}
\newcommand{\nc}{\newcommand}

\nc{\parent}[1]{$[\![#1]\!]$}
\newtheorem{theorem}{Theorem}[section]
\newtheorem{lemma}[theorem]{Lemma}
\newtheorem{example}[theorem]{Example}
\newtheorem{corollary}[theorem]{Corollary}
\newtheorem{proposition}[theorem]{Proposition}
\newtheorem{remark}[theorem]{Remark}

\newtheorem{assumption}[theorem]{Assumption}

\newenvironment{pf-main}{{\bf \sc Proof of Theorem \ref{mainresult}.}\hspace{3mm}}{\qed}

\nc{\cadlag}{c\`{a}dl\`{a}g } \nc{\ba}{\begin{array}}
\nc{\ea}{\end{array}} \nc{\be}{\begin{equation}}
\nc{\ee}{\end{equation}} \nc{\bea}{\begin{eqnarray}}
\nc{\eea}{\end{eqnarray}} \nc{\bean}{\begin{eqnarray*}}
\nc{\eean}{\end{eqnarray*}} 
\nc{\bu}{\bullet} \nc{\nn}{\nonumber}
\nc{\cA}{{\mathcal A}} 
\nc{\cB}{{\mathcal B}} 
\nc{\cC}{{\mathcal C}} 
\nc{\bfE}{\mathbf{E}}
\nc{\cD}{{\mathcal D}} 
\nc{\bbD}{\mathbb{D}}
\nc{\bbH}{\mathbb{H}}
\nc{\bbF}{\mathbb{F}}
\nc{\bbG}{\mathbb{G}}
\nc{\cG}{{\mathcal G}} 
\nc{\cF}{{\mathcal F}}
\nc{\cS}{{\mathcal S}} 
\nc{\cU}{{\mathcal U}} 
\nc{\cH}{{\mathcal H}}
\nc{\cK}{{\mathcal K}} 
\nc{\cL}{{\mathcal L}} 
\nc{\cM}{{\mathcal M}}
\nc{\cO}{{\mathcal O}} 
\nc{\cP}{{\mathcal P}} 
\nc{\cQ}{{\mathcal Q}} 
\nc{\bbE}{\mathbb{E}}
\nc{\bbEQ}{\mathbb{E}_{\mathbb{Q}}} 
\nc{\eps}{\varepsilon}
\nc{\bbEP}{\mathbb{E}_{\mathbb{P}}}
\nc{\bbL}{\mathbb{L}}
\nc{\bbP}{\mathbb{P}} 
\nc{\bbQ}{\mathbb{Q}} 
\nc{\del}{\partial}
\nc{\Om}{\Omega} 
\nc{\om}{\omega} 
\nc{\bbR}{\mathbb{R}}
\nc{\bbC}{\mathbb{C}} 
\nc{\bfr}{\begin{flushright}}
\nc{\efr}{\end{flushright}} 
\nc{\dXt}{\Delta X_{t}} 
\nc{\dXs}{\Delta
X_{s}} \nc{\bs}{\blacksquare} 
\nc{\dX}{\Delta X} 
\nc{\dY}{\Delta Y}
\nc{\dnkx}{\left(X(T^{n}_{k})-X(T^{n}_{k-1})\right)}
\nc{\esssup}{\mathrm{ess}\mbox{ }\mathrm{sup}}
\nc{\essinf}{\mathrm{ess}\mbox{ } \mathrm{inf}}
\nc{\dhats}{\widehat{\delta_s}} 
\nc{\tX}{\tilde{X}}
\nc{\tZ}{\tilde{Z}}
\nc{\what}{\widehat}
 
\nc{\half}{\frac{1}{2}}

\def\rar{\rightarrow} 
\nc{\uar}{\uparrow}
\nc{\sch}{{Schr\"odinger }}

\nc{\sgn}{\mbox{sgn}}
\nc{\chf}{\mbox{$\mathbf1$}} \nc{\eid}{\stackrel{d}{=}}

\date{\today }

\begin{document}

\title{Schr\"odinger's problem with constraints}
\author{Beatrice Acciaio}
\address{Department of Mathematics,	ETH, R\"amistrasse 101, 8092 Zurich}
\email{beatrice.acciaio@math.ethz.ch}
\author{Umut \c{C}et\.in}
\address{Department of Statistics, London School of Economics and Political Science, 10 Houghton St., London, WC2A 2AE}
\email{u.cetin@lse.ac.uk}
\maketitle

\begin{abstract}
Motivated by the connection between the Kyle equilibrium with static private signal and  the Brownian bridge, we study a much broader class of bridges that allow one to consider more general equilibrium models, for example ones including trading costs and default risk. We show that such bridges are solutions to problems of the Schr\"odinger-type. Leveraging this connection, we obtain that the equilibria in models with trading costs converge to equilibria in the classical Kyle model.\\

\end{abstract}

\section{Introduction}
In the present paper, we connect 
continuous-time models of insider trading and price formation (such as the Kyle's model) to the optimal transport theory, and show that market equilibria can be characterized as solutions to specific 2-dimensional \sch bridges. In order to be able to include defaultable assets, we study \sch bridge problems for killed time-inhomogeneous diffusion process. 

We start by recalling some basic facts about classical Kyle's model and its extension to include transaction costs, and from there introduce our approach via \sch bridges.\\

\paragraph*{\bf The classical Kyle's model} \label{s:kyle}
Introduced in \cite{Kyle}, Kyle's model is the canonical model in market microstructure to analyse dissemination of private information of traders to market prices and the associated price impact. In its continuous time formalisation by Back \cite{Back92},  the fundamental value of the traded asset is given by some random variable $V$, which will become public knowledge at some future deterministic date to all market participants.  Its price is determined in  equilibrium  by taking into account the interaction between three types of market participants: noise traders, market makers, and the insider. While the noise traders are non-strategic and have their cumulative demand  given by some Brownian motion $B$, the insider has the objective of maximizing  expected profit from trading given her advance  knowledge of $V$ and taking into account the price impact of trades. The market makers are competitive and set the price given the current and past levels of total demand, without being able to distinguish the noise demand from the informed one.  

The literature on the Kyle model and its continuous time extension is vast. We refer to the book of \c{C}etin and Danilova \cite{DMB-CD} for the background material on the classical model  and its  several extensions from the point of view of potential theory of Markov processes and stochastic filtering.  Notable works in continuous time include \cite{choRA}, \cite{D}, \cite{CDRA}, \cite{CetRH}, \cite{BP98}, \cite{CX13}, \cite{EMZsv}, \cite{ma2018kyle}, \cite{choilarsenkyle}, \cite{CSLtarget}, and \cite{QZnewkyle}. More recently, starting with the work of Bose and Ekren \cite{BERA} and \cite{BEmultdRA}, the optimal transport theory has been applied to solve some open problems in the Kyle model in particular multidimensional settings (see also Back et al. \cite{BackEkren} and \c{C}etin \cite{KpenC}). In the present paper, we also connect optimal transport to equilibrium models, but from
a different perspective, with the main goal of revisiting the Kyle-Back model from the lens of \sch bridges.

Assume that noise traders' demand is given by a Brownian motion $B$ as in \cite{Back92}, and the fundamental value of the asset is given by  $V=F(Z_1)$ for some strictly increasing function $F$ and a standard Normal random variable $Z_1$ (which is without loss of generality if $V$ has a continuous distribution). If the insider perfectly observes $V$, \cite{Back92} shows (see also Section 7 of \cite{DMB-CD} for a proof in a more general case) that, in equilibrium, the total demand process $Y$ in insider's filtration follows the dynamics
\be \label{e:kyleB}
dY_t=dB_t +\frac{Z_1-Y_t}{1-t}dt.
\ee 
In particular, $Y_1=Z_1$.  Moreover, $Y$ is a Brownian motion in its own filtration; that is, the insider maintains the same distribution for the total demand in equilibrium as that of noise.
\\

\paragraph*{\bf Including default risk} One can also consider the case of a defaultable asset within this equilibrium framework. To this end, suppose $Z$ is a Brownian motion killed at $0$, with $Z_0=z_0>0$, and the insider knows perfectly $Z_1$. The fundamental value $V$ is still given by $F(Z_1)$ with $F(0)=0$; that is, the asset pays nothing in case default happens before time $1$. Supposing that the market only observes the total demand process $Y$, one can show (see Chapter 8 of \cite{DMB-CD} for the methodology) that, in equilibrium, the total demand process follows the dynamics
\be \label{e:eqdemdef}
dY_t = dB_t +\Big(\chf_{Z_1>0}\frac{q_z(t,1;Y_t,Z_1)}{q(t,1;Y_t,Z_1)}+ \chf_{Z_1=0}\frac{L_z(t,1;Y_t)}{L(t,1;Y_t)}\Big)dt,
\ee
where $q$ is the transition density of the killed Brownian motion, $q_z$ the partial derivative of $q(\cdot,\cdot,z,\cdot)$ with respect to $z$, and $L(t, 1; Y_t)$ the likelihood of default before time $1$ given the current demand level at time $t$. We refer the reader to Chapter 8 of \cite{DMB-CD} for a more general setting of the Kyle model with default risk (see also \cite{CCdef} and \cite{CCDdef}).

In the models above, with or without default risk, the trading strategies that are not absolutely continuous are suboptimal for the insider (see \cite{Back92} or \cite{CD-GKB}). Thus, one only needs to consider strategies $\theta$ of the form $d\theta_t=\alpha_t dt$ for some appropriately measurable process $\alpha$. The total demand process is then given by $Y=B+\theta$, and in equilibrium the insider follows a bridge strategy so that the total demand process satisfies $Z_1=Y_1$. 
\\

\paragraph*{\bf Including transaction costs} The above bridge strategies may no longer be feasible if the insider incurs additional transaction costs when submitting her orders. \cite{KpenC} studies this case without default risk, i.e. $Z$ is a standard Brownian motion without killing, when the insider is subject to quadratic transaction costs; that is, the strategy
$\theta_t=\int_0^t \alpha_sds$ is subject to an additional transaction cost of
\[
\frac{\eps}{2}\int_0^1 \alpha^2_tdt
\]
for some $\eps>0$. As constructing a bridge strategy ensuring $Y_1=Z_1$ is infinitely costly under quadratic costs, these are no longer optimal strategies. Indeed, the equilibrium demand under  additonal quadratic costs is given by
\be \label{e:eqdempen}
dY_t= dB_t +\frac{\rho^\eps_y(t,Y_t,V)}{\rho^\eps(t,Y_t,V)}dt,
\ee
where
\be \label{e:rhoeqpen}
\rho^\eps(t,y,v):=\bbE\Big[\exp\Big(\frac{j^\eps(v,B_1)}{\eps}\Big)\Big|B_t=y\Big].
\ee

The function $j^\eps$ above is given by  $j^\eps(v,y):= \phi^\eps(v) +\zeta^\eps(y)+yv$, with  $\phi^\eps$ and $\zeta^\eps$
being  the solutions of the following equations:
\be \label{e:Sch1_intro}
\begin{split}
\int\exp\Big(\frac{yv+\phi^\eps(v)}{\eps}\Big)\nu(dv)=\exp\Big(-\frac{\zeta^\eps(y)}{\eps}\Big)\\
\int_{\bbR}\exp\Big(\frac{yv+\zeta^\eps(y)}{\eps}\Big) \frac{1}{\sqrt{2\pi}}\exp\Big(-\frac{y^2}{2}\Big)dy=\exp\Big(-\frac{\phi^\eps(v)}{\eps}\Big),
\end{split}
\ee
where $\nu$ is the distribution of $V$. See also Qiao and Zhang \cite{QZnewkyle} for a new approach to consider more general transaction costs at the cost of less explicit equilibrium strategies.
\\

\paragraph*{\bf
A new bridge perspective}\label{sec.conn}
In equilibria corresponding to the classical Kyle case \eqref{e:kyleB} as well as the transaction cost case \eqref{e:eqdempen},  the insider's optimal strategy is to control the process $Y$ so that $(Z_1,Y_1)$ has a particular distribution $\mu_1^*$, while $Y$ is a Brownian motion in its own filtration at the same time. In both cases, the equilibrium trading strategy of the insider is given by $d\tilde \theta_t=\tilde\alpha_t dt$, where $\tilde\alpha$ coincides with the drift of a suitable $h$-transformation  by a (space-time) harmonic function $\rho(\cdot, \cdot, v)$ for each realization of $V=v$. The function $\rho$ is the function $\rho^\eps$ given in \eqref{e:rhoeqpen} in the presence of quadratic transaction costs. Otherwise, $\rho(t,y,v)=p(1-t,y,F^{-1}(v))$, where $p$ is the transition density of a standard (un-killed) Brownian motion.  Recall that, in the aforementioned models,  $Z_1=F^{-1}(V)$ is the time-1 value of a  standard Brownian motion.
 
The perspective adopted in this paper is that of reading optimal strategies as bridges of the two-dimensional process $(Z,Y)$ so that its distribution at time $1$ matches the distribution $\mu_1^*$ that will be determined in equilibrium.  Our approach will then involve 2-dimensional \sch bridges between the distributions $\delta_{(z_0,y_0)}$ and 
$\mu_1^*$.

It is important to note that the probability measure $\mu_1^*$  is determined in equilibrium and not given exogenously. More precisely,  denoting by $\gamma$ the distribution of a standard Normal, we have
\[
\mu_1^*(dz,dy)=\mu_1^0(dz,dy):=\gamma(dz)\delta_{z}(dy)\] 
in \eqref{e:kyleB},  when there are no additional transaction costs. On the other hand, when quadratic costs are applicable, as in \eqref{e:eqdempen}, the joint distribution of $V$ and $Y_1$ becomes  
\[
\mu_1^*(dv,dy)=\mu_1^\eps(dz,dy):=\exp(\frac{j^\eps(v,y)}{\eps})\nu(dv)\frac{1}{\sqrt{2\pi}}\exp(-\frac{y^2}{2})dy.
\]
Since $\nu$ is absolutely continuous with respect to the Lebesgue measure, a simple calculation then reveals that
\begin{equation}\label{eq.mu1eps}
\mu_1^\eps(dz,dy)=
\exp\Big(\frac{j^\eps(F(z),y)}{\eps}\Big)\eta(dz) \eta(dy).
\end{equation}
Moreover, defining $\phi^\eps_F(z):=\phi^\eps(F(z))$, we can rewrite the system \eqref{e:Sch1_intro} as
\be \label{e:Sch1_alt}
\begin{split}
\int\exp\Big(\frac{yF(z)+\phi^\eps_F(z)}{\eps}\Big)\eta(dz)=\exp\Big(-\frac{\zeta^\eps(y)}{\eps}\Big)\\
\int_{\bbR}\exp\Big(\frac{yF(z)+\zeta^\eps(y)}{\eps}\Big) \eta(dy)=\exp\Big(-\frac{\phi^\eps_F(z)}{\eps}\Big).
\end{split}
\ee
We will see in Section~\ref{sect.sch} how this system  is a specific case of \sch equations.
\\

\paragraph*{\bf
Main contributions}
We consider a (possibly) killed time-inhomogeneous diffusion process
\[ 
dZ_t=\chf_{Z_t\in E}\big(b(t,Z_t)dt+\sigma(t,Z_t)dW_t\big),
\]
where killing occurs as
soon as $Z$ (that plays the role of the private signal in equilibrium models as described above) exits the set $E$, together with another process $Y$ (that plays the role of the total demand process). We study bridge problems for the pair $(Z,Y)$, introducing variations of the classical \sch problem, in particular restricting to processes having support on $E$, that is,
\begin{equation*}
\mathcal{S}_E(\mu_0,\mu_1):=\inf\left\{H(P|R): P\in\mathcal{P}(\Omega_E)\ \text{ s.t. }\, P_0=\mu_0, P_1=\mu_1\right\},
\end{equation*}
for given initial and terminal distributions $\mu_0,\mu_1$, and reference measure $R$ (see problem \eqref{P:unconstrainedeps}).
We refer to this as `unconstrained' entropy minimization problem.  
We show existence of uniqueness of its solution, as well as a characterization of the corresponding SDE whose unique weak solution has such law (Theorem~\ref{thm:first}).
We then define a `constrained' variation of it, by restricting to distributions that are solutions to a system of SDEs whose drift terms are subject to constraints in line with the assumptions of the  Kyle-Back model. This reads as
\begin{equation*}
\bar{\mathcal{S}}(\mu_0,\mu_1):=\inf\{H(P|R):P\in \bar{\cP}(\mu_1)\},
\end{equation*}
where the set  $\bar{\cP}(\mu_1)$ is a subset of $\{P\in\mathcal{P}(\Omega_E)\ \text{ s.t. }\, P_0=\mu_0, P_1=\mu_1\}$ (see problem \eqref{e:minentadpt}).
It turns out that the constrained and unconstrained problems are in fact equivalent (Theorem~\ref{thm.cunc}), so that the minimal entropy of the unconstrained problem can be achieved by using the distribution of SDE that describes the equilibrium in a Kyle-Back model with transaction costs. This in particular implies that one does not need to control the drifts of both $Z$ and $Y$ to achieve $\mathcal{S}_E(\mu_0,\mu_1)$, which is the characteristic of solutions of the \sch problem in the existing literature. 

We connect those bridge problems to equilibria in the Kyle model in the presence of penalties or transaction costs (Section~\ref{sect:constr}), showing that 
optimal bridges (solutions to the former) corresponds to optimal strategies (solutions to the latter).
When aiming for a similar connection for the case of equilibria in the classical Kyle model, i.e. without penalties, we need
to do this by approximation. This is due to the fact that the joint distribution of $(Z_1,Y_1)$ in equilibrium is singular with respect to the two-dimensional Wiener measure, leading to infinite entropy. To this end, we introduce auxiliary equilibrium problems with penalties and then let the penalty go to zero, which corresponds to considering problems $\mathcal{S}_E(\mu_0,\mu^\eps_1)$ such that $\mu^\eps_1\rightharpoonup\mu_1$ for $\eps\to 0$ (see Theorem~\ref{thm:conv}). In particular, we show that  equilibria in models with trading costs converge to equilibria of the classical Kyle model (Corollary~\ref{cor.convinsprofit}).

The proof of Theorem~\ref{thm.cunc} establishing the equivalence of constrained and unconstrained problems relies on a novel progressive enlargement formula.  This is formulated in Theorem~\ref{theorem:H}, where we show the canonical decomposition of a diffusion process when its filtration is progressively enlarged by another process under the assumption that the processes are jointly Markov.
\\

\paragraph*{\bf Organization of the paper} The rest of the paper is organized as follows. In Section~\ref{sect.sch} we recall the \sch problem and penalized optimal transport problem, pointing out the connection to the Kyle model with quadratic transaction cost. Section~\ref{sect.reg} is devoted to regularity results on killed time-inhomogeneous diffusion processes. In Section~\ref{sect:c_and_unc} we introduce the type of constrained and unconstrained bridges studied in this paper. We show existence and uniqueness of solution to the unconstrained ones, as well as a characterization of the solution to the constrained ones. Moreover, we prove that solutions to the latter give equilibria in the Kyle model with transaction costs. In Section~\ref{sect.equiv} we show equivalence between constrained and unconstrained bridges. In Section~\ref{sect.lim} we 
relate the equilibrium problem in the classical Kyle model
(without penalties) to some \sch bridge problem, studying the limit of the case with transaction costs when the costs tend to zero. Section~\ref{sect:enlarg}
 is devoted to progressive enlargement of filtrations with the filtration generated by a process.

\section{\sch problems and first connections with equilibria}\label{sect.sch}
First formalized by \sch in \cite{sch31,sch32} in the context of quantum physics, the problem that takes his name was later reformulated by F\"ollmer in modern probabilistic terms in \cite{F88}. Consider the path-space $\cC=C([0,1];\mathbb R^d)$ of continuous $\mathbb R^d$-valued
paths on the unit time interval $[0,1]$, and a probability measure $R\in\cP(\cC)$ seen as the law of some reference process.
For two given measures $\mu_0,\mu_1\in\cP(\mathbb R^d)$, the problem consists in finding a probability measure $P\in\cP(\cC)$ that minimizes the relative entropy with respect to the reference measure $R$, while matching the marginal distributions $\mu_0$ and $\mu_1$ at time $0$ and $1$, respectively. Recall that the relative entropy of a probability $p$ with respect to another probability $r$ is given by
\[
H(p|r)= 
\begin{cases}
\int \ln{\left(\frac{dp}{dr}\right)dp}, & \text{if $p\ll r$}, \\
+\infty, & \text{else},
\end{cases}
\]
and that, for any measurable function $g$, we have the decomposition
\begin{equation}\label{eq.add}
H(p|r)=H(g_\#p|g_\#r)+\int H(p(\cdot|g=z)|r(\cdot|g=z))g_\#p(dz),
\end{equation}
where the probability $g_\#p$ is the push-forward of $p$ through $g$, that is, $g_\#p(A)=p(g^{-1}(A))$ for any Borel set $A$.
Then the (dynamic) \sch bridge problem can be formulated as
\be \label{sch.pb}
\mathcal{S}(\mu_0,\mu_1):=\inf\left\{H(P|R): P\in\cP(\cC)\ \text{ s.t. } P_0=\mu_0, P_1=\mu_1
\right\},
\ee
where, for a process $X\sim P$, $P_t\in\cP(\mathbb R^d)$ is the marginal law of $X$ at time $t$, i.e., $P_t={X_t}_\#P$.

Now, take $g$ to be the projection on initial and final state of a path in $\cC$, so that, for $X\sim P$, $g_\#P=(X_0,X_1)_\#P=:P_{01}\in\cP(\mathbb R^d\times \mathbb R^d)$ is the joint law of the initial and final locations under $P$. Then the disintegration of $P$ w.r.t. $P_{01}$ is given by
\[
P(d\omega)=\int_{\mathbb R^d\times \mathbb R^d}P^{xy}(d\omega)P_{01}(dx,dy),
\]
where $P^{xy}(\cdot)=P(\cdot|X_0=x,X_1=y)$ is the bridge of $P$ between $x$ and $y$, and the decomposition in \eqref{eq.add} gives
\[
H(P|R)=H(P_{01}|R_{01})+\int_{\mathbb R^d\times \mathbb R^d}H(P^{xy}|R^{xy})P_{01}(dx,dy).
\]
One can then show that the  \sch problem \eqref{sch.pb} admits a unique solution $\hat P$ given by the following mixture of bridges $R^{xy}$ of the reference measure $R$:
\[
\hat P(d\omega)=R^{xy}(d\omega)\hat{\pi}(dx,dy),
\]
where $\hat\pi$ is the unique solution to the (static) \sch bridge problem
\begin{equation}\label{sch.stat}
s(\mu_0,\mu_1):=\inf\left\{H(\pi|R_{01}): \pi\in\Pi(\mu_0,\mu_1)\right\},
\end{equation}
with $\Pi(\mu_0,\mu_1)=\{
\pi\in\cP(\mathbb R^d\times\mathbb R^d)\ \text{ s.t. } \pi_0=\mu_0, \pi_1=\mu_1\}$;
see \cite{F88} and \cite{leonard2012}. 

Now, recall that, given two distributions $\mu,\nu\in\cP(\mathbb R^d)$, a measurable function $c$ on $\mathbb R^d\times\mathbb R^d$, and a regularization parameter $\eps\geq 0$, the entropically regularized optimal transport problem can be formulated as
\begin{equation}\label{eot}
\text{EOT}_\eps(\mu,\nu,c):=\inf_{\pi\in\Pi(\mu,\nu)}\int c(x,y)\pi(dx,dy)+\eps H(\pi|\mu\otimes\nu),
\end{equation}
where the case of $\eps=0$ corresponds to the classical transport problem ($\text{EOT}_\eps\equiv\text{OT}$).
When $\eps>0$, the regularization term makes the problem strictly convex, so that \eqref{eot} admits a unique solution, that we denote by $\pi^\eps(\mu,\nu,c)$.
Note that problem \eqref{eot} can be recast as static \sch problem $s(\mu,\nu)$ with $R_{01}$ in \eqref{sch.stat} being replaced by $e^{-c/\eps}d(\mu\otimes\nu)$ appropriately normalized. Or viceversa, problem \eqref{sch.stat} with $R_{01}=\mu_0\otimes\mu_1$ corresponds to the regularized transport problem $\text{EOT}_1(\mu_0,\mu_1,0)$.

Notably, the optimizer $\pi^\eps=\pi^\eps(\mu,\nu,c)$ can be described in terms of the so-called Schr\"odinger potentials, which are functions 
$\varphi^\eps, \psi^\eps$ on $\mathbb R^d$ such that
\begin{equation}\label{eq:denspieps}
\frac{d\pi^\eps}{d\mu\otimes\nu}(x,y)=e^{\frac{\varphi^\eps(x)+\psi^\eps(y)-c(x,y)}{\eps}}\quad \mu\otimes\nu\text{-a.s.}
\end{equation}
and found as (unique up to translation) solution to the following system of equations:
\begin{align}
\begin{split}\label{sch.pot}
\varphi^\eps(x) =& - \eps\log  \int \exp\Big(\frac{\psi^\eps(y) - c(x,y)}{\eps}\Big) \, \nu(dy) \quad \mu\text{-a.s.}\\
\psi^\eps(y) =& -\eps \log  \int \exp\Big(\frac{\varphi^\eps(x) - c(x,y)}{\eps}\Big) \, \mu(dx) \quad \nu\text{-a.s.}
\end{split}
\end{align} 

In order to ease the connection with equilibrium problems, consider, for $d=1$, the entropic transport problem
\eqref{eot} with {\em quadratic} cost $c_2(x,y):=\frac12(F(x)-y)^2$. Then, for every $\pi\in\Pi(\mu,\nu)$, we have 
\[
\int c_2(x,y)\pi(dx,dy)=\frac12 \int F^2(x)\mu(dx)+\frac12 \int y^2\nu(dy)-\int F(x)y \pi(dx,dy),
\]
so that the first two terms on the right hand side do not depend on the specific coupling $\pi$, and thus this corresponds to considering the cost $c(x,y)=-F(x)y$.

This shows how the pair  $(\phi^\eps_F,\zeta^\eps)$ from \eqref{e:Sch1_alt}, i.e. the equations characterising the Kyle equilibrium with quadratic transaction cost, are related to the
Schr\"odinger potentials $(\varphi^\eps, \psi^\eps)$ in \eqref{sch.pot} associated to the regularized transport problem $\text{EOT}_\eps(\eta,\eta,c_2)$, where $\mu=\eta$ (distribution of $Y_1$)
and $\nu=\eta$ (distribution of $Z_1$):
\begin{equation}\label{eq:relphipsi}
 \phi_F^\eps(x)=\varphi^\eps(x)-\frac{F(x)^2}{2},\qquad \zeta^\eps(y)=\psi^\eps(y)-\frac{y^2}{2}.  
\end{equation}
Note that the corresponding optimizer $\pi^\eps(\eta,\eta,c_2)$ coincides with the joint distribution $\mu_1^\eps$ of $Z_1$ and $Y_1$ given in \eqref{eq.mu1eps}; see also the discussion at the end of Section~\ref{sect:c_and_unc}.

\section{Regularity results on Killed SDEs and setup}\label{sect.reg}
We consider the (possibly killed) time-inhomogeneous diffusion process $Z$ as the weak solution of
\begin{equation}\label{eq:Zref}
dZ_t=\chf_{Z_t\in E}\big(b(t,Z_t)dt+\sigma(t,Z_t)dW_t\big),\; t\in [0,1], \quad Z_0=z_0 \in E, 
\end{equation}
where $E=(\ell, \infty)$, with $\ell \in [-\infty, \infty)$, and $W$ is a Brownian motion. The killing occurs as soon as  $Z$ exits the set $E$, after which it remains constant. The assumption below on the drift and diffusion coefficients implies that, when taking the state space $E=\bbR$, one obtains a diffusion that is never killed. On the other hand,  any boundary  $\ell  \in (-\infty, \infty)$ can be reached in finite time from any initial point. That is, killing occurs if and only if  $\ell$ is finite. To unify the notation, we set $E_{\ell}:=E\cup {\ell}$ if $\ell$ is finite, and $E_{\ell}=E$ otherwise.

\begin{assumption}\label{a:main} The functions $[0,1]\times E \ni (t,z) \mapsto b(t,z)$ and $[0,1]\times E \ni (t,z) \mapsto \sigma(t,z) \in [\underline{\sigma},\infty)$, for some $\underline{\sigma}>0$, are bounded, Lipschitz in $z$ (uniformly in $t$), and H\"older continuous in $t$ (uniformly in $z$). 
\end{assumption}
Under the above assumption, there exists a unique strong solution to  \eqref{eq:Zref} for any $z_0\in E$. Moreover, $Z$ is a strong Markov process by \cite[Remark 2.2.4]{DMB-CD}. For $s\in[0,1]$ and $z\in E$, we denote by $\bbP^{s,z}$ the law of $(Z_t)_{t\in[s,1]}$ conditioned on $\{Z_s=z\}$.  
The following proposition collects some results regarding the transition function of $Z$.

\begin{proposition}\label{prop.reg} Under Assumption \ref{a:main}, the following statements hold:
\begin{enumerate}
\item \label{qexists} For any $t\in (0,1], y\in E $, there exists $q(\cdot, t; \cdot, y) \in C^{1,2}([0,t)\times E)$ such that
\begin{align} \label{e:defL}
\begin{split}
\bbP^{s,z}(Z_t\in dy)&=q(s,t;z,y)dy,  \quad s<t,\, z\in E\quad \text{and}\\
L(s,t;z)&:=\bbP^{s,z}(Z_t=\ell)=1-\int_E q(s,t;z,y)dy, \quad s<t,\,\, z\in E. 
\end{split}
\end{align}
Moreover, if $\ell$ is finite, for all $t>s$ and $y\in E$ we have
\[
\lim_{z\rar \ell}q(s,t;z,y)=0.
\]
\item \label{qintUB} The transition density $q$ and its derivatives satisfy the following bounds:
\[
|D_s^\alpha D_z^\beta q(s,t;z,y)|\leq c(t-s)^{-\frac{1+2\alpha +\beta}{2}}\exp\left(-C\frac{(y-z)^2}{t-s}\right),
\]
for $2\alpha+\beta \leq 2$, where  $c$ and $C$ are constants independent of $s,t,z$ and $y$, and $D_s^{\alpha}$ (resp. $D_z^{\alpha}$) denotes the $\alpha$-th partial derivative with respect to $s$ (resp. $z$).

\item \label{ball} For any $r>0$, $y\in E$, and $t>0$, 
\begin{align}
&\sup_{\substack{z\in B_r^c(y)\\s<t}}q(s,t;z,y)<\infty,\quad \text{and}\\
&\lim_{s\rar 0}\int_{B_r^c(z)}q(0,t-s;x,y)q(t-s,t;y,z)dy=0, \quad \forall x\in E, \label{e:tdduality0}
\end{align}
where $B_r(y):=\{z\in E:|y-z|<r\}$.
\end{enumerate}
\end{proposition}
\begin{proof}
Items \eqref{qexists} and \eqref{qintUB} follow from Sections~IV.11 and~IV.13 (when $\ell=-\infty$) and Theorem~IV.16.3 (when $\ell$ is finite) in Ladyzhenskaya et al.~\cite{ladyzhenskaia}.
Item \eqref{ball} is a direct consequence of these bounds by means of the dominated convergence theorem. 
\end{proof}

\begin{remark}
    In the sequel, we only use the properties of the transition density $q$ that are established in Proposition \ref{prop.reg}. Thus, all of the following continue to hold as long as the coefficients of the SDE \eqref{eq:Zref} are regular enough to ensure the validity of that proposition. 
\end{remark}

We shall denote by $\eta$ the distribution of $Z_1$ under $\bbP^{0,z_0}$, that is, 
\[
\eta(dz)=\chf_{z\in E} q(0,1;z_0,z) dz+L(0,1;z_0)\delta_{\ell}(dz).
\]
In particular, $\eta$ is  a measure on the Borel subsets of $E$ (or  $E_{\ell}$ if $\ell$ is finite).
We also set $\partial E^2_\ell:=(E_\ell\times\{\ell\})\cup (\{\ell\}\times E_\ell)$, the boundary of $E^2_\ell$. A straightforward application of the dominated convergence theorem yields the following regularity result on the transition densities.

\begin{corollary}\label{cor.reg} Under the hypothesis of Assumption \ref{a:main}, the following results hold:
\begin{enumerate}
\item \label{qsmooth} For any  $g:E\to\bbR$ such that  $\log g$ has  subquadratic growth,  and for $t \leq 1$, the map 
\be \label{e:Qtf}
[0,t)\times E\ni (s,z) \mapsto Q_sg(z):=\int_E g(y)q(s,t;z,y)dy  
\ee
is in $C^{1,2}([0,t)\times E)$. In particular, by taking $g\equiv 1$, $L(\cdot,t,\cdot)\in C^{1,2}([0,t)\times E)$.

Similarly, for any  $g:E\times E\to\bbR$ such that  $\log g$ has  subquadratic growth, and for $t\leq 1$, the map
\be \label{e:Qtff}
[0,t)\times E^2\ni (s,x) \mapsto P_sg(x):=\int_{E^2} g(z',y')p(s,t;x,(z',y'))dy'dz'  
\ee
is in $C^{1,2}([0,t)\times E^2)$,
where, for all $0\leq s <t \leq 1$ and $x=(z,y)\in E^2 ,(z',y') \in E^2$, 
\be \label{e:defp}
p(s,t; (z,y), (z',y')):=q(s,t;z,z')q(s,t;y,y').
\ee
\item
If $\ell$ is an accessible boundary, i.e. $\ell\in(-\infty,+\infty)$, then, for all $s\in [0,t)$, we have 
\[
\lim_{z\rar \ell}Q_sg(z)=0,\quad \text{and}\quad \lim_{x\rar \partial E^2_\ell}P_sg(x)=0. 
\]
\end{enumerate}
\end{corollary}

\section{Constrained and unconstrained \sch bridges}\label{sect:c_and_unc}
From this section on, we will  study \sch bridge problems for (possibly) killed time-inhomogeneous diffusion process and relate them to equilibrium problems. For this, we will adopt the following canonical setting. 
Let $\Om_E=C([0,1]; E_\ell^2)$  be the space of continuous  $E_\ell^2$-valued maps on $[0,1]$  absorbed at $\ell$. That is, $\om=(\om^1,\om^2)\in \Om_E$ only if $\om^i_t=\ell$ for any $t>\zeta^i:=\inf\{t\geq 0: \om^i_t=\ell\}$. We also denote by  $X=(X^1,X^2)$ the coordinate mapping and let  $\bbF=(\cF_t)_{t\in[0,1]}$ correspond to the canonical filtration on $\Om_E$. The coordinates $X^1, X^2$ will play the role of the private signal $Z$ and of the demand process $Y$, respectively, in an equilibrium problem  as in Section \ref{s:kyle}.
We will consider as a reference measure $R$ the joint law of two independent solutions of \eqref{eq:Zref}, that is, $R=law(Z)\otimes law(Z)$.  This in particular entails $R_1= \eta\otimes \eta$, with $\eta$ being the law of $Z_1$ under $\bbP^{0,z_0}$. 

For the \sch problem \eqref{sch.pb} 
to be well-defined, the target measure $\mu_1$ needs to be absolutely continuous w.r.t. $R_1$. More precisely, to avoid the trivial case we will consider
\begin{equation}\label{eq.finen}
\mu_1\in\mathcal{P}(\bbR\times\bbR)\; \text{ s.t. }\; H(\mu_1|R_1)<\infty.
\end{equation}  
We introduce the Radon-Nikodym density of $\mu_1$ with respect to $R_1$, and of its marginals with respect to $\eta$, as 
\begin{equation}\label{eq.dens}
f=\frac{d\mu_1}{dR_1},\quad f_1(z)= \int_{E_{\ell}}f(z,y)\eta(dy), \quad  f_2(y)= \int_{E_{\ell}}f(z,y)\eta(dz).
\end{equation}
For most of the results of this paper, we will rely on the following technical assumption.
\begin{assumption}\label{ass:f}
The density function $f$ is continuous and such that $\log f$ has subquadratic growth.
\end{assumption}

Beside the classical (unconstrained) \sch problem, we will also consider a modification that we call {\em constrained} \sch problem, motivated by equilibrium considerations in the Kyle model. Indeed, in equilibrium the insider conditions the demand process to end up at a certain level when the trading stops.  
Since the private signal $Z$ of the insider should be viewed as her estimate of the fundamental value of the traded asset and is, therefore, given and cannot be altered, the aforementioned conditioning can only be achieved by adding a suitable drift to the demand process $Y$. 
This will lead to a modification of the \sch problem, where the optimization is restricted to laws of such pairs of processes $(Z,Y)$  (see Section~\ref{sect:constr} below).

\subsection{Unconstrained \sch bridges}\label{sect:unconstr}
In this section, we fix marginals $\mu_0=\delta_{(z_0,z_0)}$ and $\mu_1$ satisfying \eqref{eq.finen}, and reference measure $R=law(Z)\otimes law(Z)$, and
study a variant of the classical \sch bridge problem \eqref{sch.pb} by restricting to probabilities having support in $\Omega_E$; that is, 
\be \label{P:unconstrainedeps}
\mathcal{S}_E(\mu_0,\mu_1):=\inf\left\{H(P|R): P\in\mathcal{P}(\Omega_E)\ \text{ s.t. }\, P_0=\mu_0, P_1=\mu_1\right\}.
\ee
In order to describe its solutions, we set,  for $t\in[0,1]$ and $(z,y)\in E^2_{\ell}$,
\be\label{e:hepsdef}
h(t,(z,y)):=\bbE^R\bigg[f(X^1_1,X^2_1) \Big|(X^1_t,X^2_t)=(z,y)\bigg],
\ee
where $f$ is the density in \eqref{eq.dens}.
We will use the following notation: 
\[
\mathbf{b}(t, (z,y))=\begin{bmatrix} b(t,z)\\
b(t,y)
\end{bmatrix}
,\quad  {\Sigma}(t, (z,y))=\begin{bmatrix} \sigma(t,z) & 0\\
0& \sigma (t,y) \end{bmatrix}, \quad \mathbf{I}(t, (z,y))=\begin{bmatrix} \chf_E(z)\\
\chf_E(y)
\end{bmatrix}.
\]
The next theorem extends   \cite[Theorem 3.1]{Dai91} and \cite[Sect. II.1.2-1.3]{F88}.
\begin{theorem}  \label{thm:first} 
Under Assumptions~\ref{a:main} and \ref{ass:f}, the SDE 
\begin{equation}
\label{e:epsbridge_g}
\begin{split}
X_t &=x_0+\int_0^t \mathbf{I}(u, X_u)\cdot
\left\{\mathbf{b}(u,X_u) + \Sigma^2(u,X_u)\frac{(\nabla h(u,X_u))^*}{h(u,X_u)}\right\} du\\
&\hspace{1cm}+\int_0^t \mathbf{I}(u, X_u)\cdot\Sigma(u,X_u)d\beta_u
\end{split}
	\end{equation}
 admits a unique weak solution,	where $A^*$ stands for the transpose of matrix $A$, $x_0=(z_0,z_0)$, and $\beta$ is a 2-dimensional standard Brownian motion. The corresponding   law of the solution coincides with the unique minimiser $P^*$ of the \sch problem $\mathcal{S}_E(\mu_0,\mu_1)$ in \eqref{P:unconstrainedeps}. In particular, for every $t\in [0,1]$, $\frac{dP^*}{dR}|_{\cF_t}=h(t,X_t)$.
\end{theorem}

\begin{proof}
Note that Assumption \ref{a:main} in conjunction with Theorem A.1 in \cite{CDbridge} implies that Assumption 2.1 in \cite{CDbridge} is satisfied. Then, Theorem 2.1 in  \cite{CDbridge} shows the existence of a unique weak solution to \eqref{e:epsbridge_g}, that we denote by $\hat P$, since $h\in C^{1,2}([0,1)\times E^2)$. 
Note that the hypotheses in  \cite{CDbridge} include smoothness of the $h$-function at the boundary point as well. However, this is rather for convenience than necessity in the application of It\^o's formula. Indeed, by localising if necessary, one can show that the semimartingale dynamics of $h(\cdot,X)$ are in the required form used in the proofs  therein. 
Moreover, by the same theorem,  under $\hat P$, $X_1$ is distributed according to $\mu_1$
and  $X$ is a strong Markov process with transition density $p^h$ such that 
\[
\hat P(X_t\in dx'|X_s=x)=p^h(s,x;t,x')dx',  \quad s<t, \, x,x'\in E^2, 
\]
where $p^h(s,x;t,x')=\frac{h(t,x')}{h(s,x)} p(s,t;x,x')$ and $p$ is as defined in \eqref{e:defp}. In particular, $p^h$ inherits the same smoothness as $p$.
Moreover, it follows from Girsanov's theorem and the weak uniqueness of the solutions of \eqref{e:epsbridge_g} that $\frac{d\hat P}{dR}|_{\cF_t}=h(t,X_t)$.

Now we claim that $\hat P=P^*$; that is, that the weak solution to (\ref{e:epsbridge_g}) has the same bridges as $R$. To this end, first observe that Assumption 2.2 in \cite{CDbridge} is satisfied. Indeed, note that \eqref{e:Qtf} implies that $Z$ is a Feller process, and in particular $(\bbP^{s,z})_{s\in [0,1], z\in E_\ell}$ is a weakly continuous family of measures  (see, e.g., Theorem 3.1 in Chapter 1 of \cite{Pinsky}). To show weak continuity of the family $(\hat P^{s,x})_{s\in [0,1], x\in E_\ell^2}$, where $\hat P^{s,x}$ denotes the law of the process $(X_t)_{t\in [s,1]}$ under $\hat P$ given $X_s=x$, it suffices to show continuity of $(s,x) \mapsto \bbE^{s,x}[G((X_t)_{t\in [s,1]})f(X_1)]$ for any bounded and continuous function $G$ acting on the set of $E_\ell^2$-valued continuous maps on $[s,1]$, where  $\bbE^{s,x}$ denotes the conditional expectation under $R$ given $X_s=x$.

To this end, let $N>0$ be arbitrary.  Then,
\begin{align*}
    \bbE^{s,x}[G((X_t)_{t\in [s,1]})f(X_1)]
    &=\bbE^{s,x}[G((X_t)_{t\in [s,1]})(f(X_1)\wedge N)]\\
    &+\bbE^{s,x}[G((X_t)_{t\in [s,1]})(f(X_1)-f(X_1)\wedge N)].
\end{align*}
Since $f\wedge N$ is continuous and bounded, it follows from the weak continuity of $(\bbP^{s,z})_{s\in [0,1], z\in E_\ell}$ that 
\[
\lim_{(p,y)\to (s,x)}\bbE^{p,y}[G((X_t)_{t\in [p,1]})(f(X_1)\wedge N)]=\bbE^{s,x}[G((X_t)_{t\in [s,1]})(f(X_1)\wedge N)].
\]
On the other hand, for large enough $N$, for any $\alpha>0$ there exists $K>0$ such that $f(x) \leq K \exp(\alpha \|x\|_2^2)$ on $\{x:f(x)>N\}$, where $\|x\|_p$ denotes the $L^p$ norm of $x$. Thus,
\begin{align*}
    \bbE^{s,x}[|G((X_t)_{t\in [s,1]})|&f(X_1)\chf_{[f(X_1)>N]}]\leq K C_G  \bbE^{s,x}[\exp(\alpha \|X_1\|_2^2)\chf_{[\|X_1\|_2^2> \frac{1}{\alpha}\log \frac{N}{K}]}]\\
    &\leq K C_G  \bbE^{s,x}[\exp(2 \alpha (\|X_1-X_s\|_2^2+\|X_s\|_2^2)\chf_{\{\|X_1-X_s\|_2^2> \frac{1}{2\alpha}\log \frac{N}{K}-\|X_s\|^2\}}],
\end{align*}
where $C_G$ is an upper bound for $|G|$. 
Thus, from Proposition~\ref{prop.reg}, and using the  Gaussian bound on the transition density, it follows that 
\[
\sup_{s\in [0,1]}\sup_{y \in B_r(x)}\bbE^{s,y}[|G((X_t)_{t\in [s,1]})|f(X_1)\chf_{[f(X_1)>N]}] \to 0
\]
as $N\to \infty$, for any $r>0$. An analogous argument also shows that 
\[
\sup_{s\in [0,1]}\sup_{y \in B_r(x)}\bbE^{s,y}[|G((X_t)_{t\in [s,1]})N\chf_{[f(X_1)>N]}] \to 0
\]
as $N\to \infty$, for any $r>0$, where $B_r(x)=\{y\in E^2: \|y-x\|\leq r\}$. This establishes the weak continuity of the family  $(\hat P^{s,x})_{s\in [0,1], x\in E_\ell^2}$.

Thus, in view of Assumption \ref{a:main} and the continuity of $h$, all that remains to be shown is that\footnote{Although the assumption therein is stated for time-homogeneous diffusions, translation of the conditions to our setting is trivial by considering the space-time process, i.e. adding time as another dimension to the state space, since there is no uniform ellipticity assumption in \cite{CDbridge}.}
\be \label{e:tddualityeps}
\lim_{s\rar 0}\int_{B_r^c(x')}p^h(0,t-s;x,y)p^{h}(t-s,t;y,x')dy=0, \quad \forall x\in E^2.
\ee
However, this follows from \eqref{e:tdduality0} since \[
p^h(0,t-s;x,y)p^h(t-s,t;y,x')=\frac{h(t,x')}{h(0,x)}p(0,t-s;x,y)p(t-s,t;y,x').
\]
Thus, if we condition on $\{X_1=x\}$, Theorem 2.2 in \cite{CDbridge} shows that the resulting bridge under $\hat P$ is the weak solution of 
\be \label{e:epsbridge2}\begin{split}
    dX_t&=\mathbf{I}(t, X_t)\cdot\Sigma(t,X_t)d\beta_t \\&+ \mathbf{I}(t, X_t)\cdot
\left\{\mathbf{b}(t,X_t)+ \Sigma^2(t,X_t)\left(\frac{(\nabla h(t,X_t))^*}{h(t,X_t)}+\frac{(\nabla p^{h}(t,X_t;1,x))^*}{p^{h}(t,X_t;1,x)}\right)\right\}dt.
\end{split}
\ee

Next observe from the definition of $p^h$ that
\[
\frac{(\nabla p^h(t,x;1,z))^*}{p^h(t,x;1,z)} =\frac{(\nabla p(t,1;x,z))^*}{p(t,1;x,z)}-\frac{(\nabla h(t,x))^*}{h(t,x)}.
\]
Thus, \eqref{e:epsbridge2} is in fact
\[
\begin{split}
    dX_t&=\mathbf{I}(t, X_t)\cdot\Sigma(t,X_t)d\beta_t \\&+ \mathbf{I}(t, X_t)\cdot
\left\{
\mathbf{b}(t,X_t)+ \Sigma^2(t,X_t)\frac{(\nabla p(t,X_t;1,x))^*}{p(t,X_t;1,x)}
\right\}dt,
\end{split}
\]
which is the SDE for $R$ conditioned on $\{R_1=x\}$ by Theorem 2.2 in \cite{CDbridge}. Note that Assumption 2.2 of \cite{CDbridge} is valid in this case due to Assumption \ref{a:main}, by analogous --and in fact simpler-- considerations.
\end{proof}

This leads to the following conditional distribution property for the solution to \eqref{P:unconstrainedeps}.
\begin{corollary}\label{c:conddist} Assume the setting of Theorem \ref{thm:first}. Then, for $i,j\in\{1,2\}, i\neq j$: 
\[
P^*(X^i_t \in dx_i|\cF^{X^j}_t)= \frac{h^{(i)}(t,x_i, X^j_t)R(X^i_t\in dx_i)}{\pi^i(t,X^j_t)}, \quad t\in[0,1],\; x_i\in E_{\ell},
\]
where  $\bbF^{X^j}=(\cF_t^{X^j})_{t\in[0,1]}$ is the minimal filtration generated by $X^j$ satisfying the usual hypothesis,
$h^{(i)}(t,u,v)$ equals $h(t,(u,v))$ (resp. $h(t,(v,u))$) if $i=1$ (resp. $i=2$), and 
\[
\pi^i(t,x_j)=\int_{E_{\ell}}h(t,x)R(X^i_t\in dx_i),
\]
with $x=(x_1,x_2)$.
\end{corollary}
\begin{proof} Without loss of generality, we show the statement for $i=1$.
Let $g$ be a bounded measurable function on $E_{\ell}$ and $A\in \cF^{X^2}_t$. Recall from Theorem \ref{thm:first} that 
$\frac{d P^*}{dR}|_{\cF_t}=h(t,X_t)$. Thus, denoting the expectation with respect to $ P^{*}$ by $\bbE^*$, we obtain
\[
\begin{split}
\bbE^*[g(X^1_t)\chf_A]&=\bbE^R[g(X^1_t)\chf_A h(t,(X^1_t,X^2_t))]\\
&=\bbE^R\left[\chf_A \int_{E_{\ell}}g(z) h(t,(z,X^2_t))R(X^1_t\in dz)\right]\\
&=\bbE^R\left[\chf_A \frac{\int_{E_{\ell}}g(z) h(t,(z,X^2_t))R(X^1_t\in dz)}{\bbE^R[h(t,X_t)|\cF^{X^2}_t]}h(t,X_t)\right]\\
&=\bbE^*\left[\chf_A \frac{\int_{E_{\ell}}g(z) h(t,(z,X^2_t))R(X^1_t\in dz)}{\bbE^R[h(t,X_t)|\cF^{X^2}_t]}\right],
\end{split}
\]
where the second equality follows from the independence of $X^1$ and $X^2$ under $R$. Moreover,
\[
\bbE^R[h(t,X_t)|\cF^{X^2}_t]=\int_{E_{\ell}}h(t,(z,X^2_t))R(X^1_t\in dz).
\] 
\end{proof}

The above immediately gives the dynamics of $X^i$, $i=1,2$, in their own filtration under $P^*$.
\begin{corollary}\label{c:natdyn}
Assume the setting of Theorem \ref{thm:first}. Then the following dynamics hold for $t\in [0,1]$:
\begin{align}
    dX^1_t&= \chf_{X^1_t\in E}\left(\sigma(t,X^1_t)dB_t^1 + b(t,X^1_t)dt+\sigma^2(t,X^1_t)\frac{\pi^2_y(t,X^1_t)}{\pi^2(t,X^1_t)}dt\right)\label{e:Schd1NF}\\
    dX^2_t&=\chf_{X^2_t\in E}\left(\sigma(t,X^2_t)dB_t^2 + b(t,X^2_t))dt+\sigma^2(t,X^2_t)\frac{\pi^1_y(t,X^2_t)}{\pi^1(t,X^2_t)}dt\right),\label{e:Schd2NF}
\end{align}
where, for $i=1,2$, $B^i$ is a $(P^*,\cF^{X^i})$-Brownian motion, $\pi^i$ is the function defined in Corollary~\ref{c:conddist}, and $\pi^i_y$ stands for the derivative of $\pi^i$ w.r.t. its second argument.
\end{corollary}
\begin{proof}
We shall only prove the dynamics for $X^2$ as the other is proved similarly. 
It follows from the standard filtering theory (see \cite[Theorem 8.1]{lipcer2001statistics}) that, on the set $\{X^2_t \in E\}$,
\[
dX^2_t= \sigma(t,X^2_t)dB_t^2 + \left(b(t,X^2_t)+\sigma^2(t,X^2_t)\bbE^*\left[\frac{h_2(t,(X^1_t,X^2_t))}{h(t,(X^1_t,X^2_t))}\Big|\cF^{X^2}_t\right]\right)dt,
\]
where $\bbE^*$ is expectation with respect to $P^*$, and  $h_i$ is the partial derivative of $h$ with respect to  $x_i$. Now, in view of Corollary \ref{c:conddist}, the conditional expectation above is given by
\[
\int_{E_{\ell}} h_2(t,(z,X^2_t))\frac{R(X^1_t\in dz)}{\pi^1(t,X^2_t)}=\frac{\pi^1_y(t,X^2_t)}{\pi^1(t,X^2_t)}.
\]
Note that differentiation under the integral sign that paves the way to the last equality is justified  in view of Proposition \ref{prop.reg} and the growth assumption on $\log f$. This completes the proof.
\end{proof}
\begin{remark}
The above result is not surprising, as the change of measure density that determines the law of $X^i$ under $P^*$ is given by $\bbE^R[h(t,X_t)|\cF^{X^i}_t]=\pi^j(t,X^i_t)$. Since the quadratic variation of $X^i$ remains unchanged, 
\[
d\pi^j(t,X^i_t)=\pi^j_y(t,X^i_t)\sigma(t,X^i_t)d\tilde{B}^i_t,
\]
for some Brownian motion $\tilde{B}^i$, and one can conclude via  Girsanov's theorem.
\end{remark}

\subsection{Constrained \sch bridges and relation with equilibrium problems with penalties}\label{sect:constr}
The aim of this section is to recover the equilibrium in the Kyle model with penalties as solution to a specific constrained \sch problem. With this in mind, in the setting introduced at the beginning of Section~\ref{sect:c_and_unc}, we interpret the canonical process $X=(X^1,X^2)$ as the pair of private signal of the insider ($X^1$) and demand process ($X^2$). Since the insider knows the terminal value $X^1_1$ in addition to observing the demand process, we will consider the filtration $\bar{\bbG}=(\bar{\cG}_t)_{t\in[0,1]}$ generated by $(X_t)_{t\in [0,1]}$ and $X^1_1$ and satisfying the usual conditions.
As a result of this filtration enlargement, the dynamics of the insider's signal and the demand process take the following form:
\begin{align}
\begin{split}\label{e:SchKyle}
dX^1_t =&\chf_{X^1_t\in E}\left\{b(t,X^1_t)dt +\sigma(t,X^1_t)d\beta^1_t\right\}\\
&+\chf_{X^1_t\in E}\sigma^2(t,X^1_t)\bigg(\frac{q_z(t,1;X^1_t,X^1_{1})}{q(t,1;X^1_t,X^1_{1})}\chf_{X^1_{1}\in E}+ \frac{L_z(t,1;X^1_t)}{L(t,1;X^1_t)}\chf_{X^1_{1}=\ell}\bigg)dt,\quad  X^1_0=z_0,\\
dX^2_t =&\chf_{X^2_t\in E}\{b(t,X^2_t) + \alpha_t\}dt + \chf_{X^2_t\in E}\sigma(t,X^2_t)d\beta^2_t, \quad X^2_0=z_0,
\end{split}
\end{align}
where $\beta=(\beta^1,\beta^2)$ is a $\bar{\bbG}$-Brownian motion, $\alpha$ is $\bar{\bbG}$-adapted, $q$ and $L$ are the functions defined in \eqref{e:defL}, and $q_z$ (resp. $L_z$) is the derivative of $q(t,1;z,y)$ (resp. $L(t,1;;z)$) with respect to $z$. In the above, the dynamics of $X^1$ simply follows from the formula in Theorem~\ref{theorem:H}\footnote{To apply the formula therein, take $Z=X^1$ and $Y_t= X^1_1$ for all $t$.}. Note that the dynamics of $X^2$ remain unchanged under the enlarged filtration since the driving Brownian motion of $X^2$ under the natural filtration $(\cF_t)_{t\in [0,1]}$ is independent of $X^1_1$. 

As we are interested in studying the \sch problem constrained to solutions of this system, we restrict our attention to the case of $\alpha$ satisfying the integrability condition
\begin{equation}\label{eq:alphasqint}
\bbE\bigg[\half\int_0^{1\wedge \zeta(X^2)}\frac{\alpha^2_t}{\sigma^2(t,X^2_t)}dt
\bigg]<\infty,
\end{equation}
where $\zeta(X^2):=\inf\{t\geq 0: X^2_t\notin E\}$.
In particular, by Girsanov theorem, this guarantees the existence of a unique weak solution to \eqref{e:SchKyle}. We denote by $P(\alpha)$ the joint law of $(X^1,X^2)$, and observe that the value of the expectation in \eqref{eq:alphasqint}
corresponds to the value of the relative entropy $H(P(\alpha)|R)$.

Formally, for $\mu_0=\delta_{(z_0,z_0)}$ and $\mu_1$ as in \eqref{eq.finen}, we consider the constrained \sch problem
\be \label{e:minentadpt}
\bar{\mathcal{S}}(\mu_0,\mu_1):=\inf\{H(P|R):P\in \bar{\cP}(\mu_1)\},
\ee
where 
\[
\bar{\cP}(\mu_1)=\{P(\alpha) : \, \text{ $\alpha$ is $\bar{\bbG}$-adapted, satisfies \eqref{eq:alphasqint}, and}\  P_1(\alpha)=\mu_1\}.
\]
That is, \eqref{e:minentadpt} corresponds to the \sch problem  \eqref{P:unconstrainedeps} with the additional constraint that $P$ is in $\bar\cP(\mu_1)$. Note that, in order for $\bar\cP(\mu_1)$ to be non-empty, the first component of $\mu_1$ must be $\eta$, which will be our standing assumption for this section. As a result, the density $f_1$ in \eqref{eq.dens} equals $1$.

\begin{theorem}\label{thm_cons}
Under Assumptions~\ref{a:main} and \ref{ass:f},  the optimizer of the constrained \sch problem \eqref{e:minentadpt} is $P({\hat\alpha})$, with $\hat\alpha_t=\sigma^2(t,X^2_t)\frac{\rho_y(t,X^2_t;X^1_1)}{\rho(t,X^2_t;X^1_1)}$, where $\rho:[0,1]\times E_\ell^2\to\bbR_+$ is defined as
\begin{equation}\label{eq:rho}
\rho(t,y;z):=\int_{E}f(z,u)q(t,1;y,u)du + f(z,\ell)L(t,1;y),
\end{equation}
with $\rho_y$ being its partial derivative with respect to $y$.
\end{theorem}
\begin{proof}

Note that $\rho(1,X^2_1;X^1_1)= {f(X^1_1,X^2_1)}$ and $\rho(0,X^2_0;X^1_1)= f_1(X^1_1){=1}$. 
Therefore, for every $P\in\bar\cP(\mu_1)$, $\bbE^{P}[\log \rho(1,X^2_1;X^1_1)]=H(\mu_1|R_1)<\infty$ 
by \eqref{eq.finen}, {and $\bbE^{P}[\log \rho(0,X^2_0;X^1_1)]=0$.}
Thanks to this and to \eqref{eq:alphasqint}, we can  reformulate problem
\eqref{e:minentadpt}
as
\[
{\bar{\mathcal{S}}(\mu_0,\mu_1)=\inf_{P(\alpha) \in\bar\cP(\mu_1) }\bbE^{P(\alpha)}}\bigg[\half\int_0^{1\wedge \zeta(Y)}\frac{\alpha_t^2}{\sigma^2(t,X^2_t)}dt-
\log \rho(1,X^2_1;X^1_1)\bigg].
\]
Note that, for each $z\in E_\ell$, Corollary~\ref{cor.reg} yields $\rho(\cdot,\cdot;z)\in C^{1,2}([0,1)\times E)$ and that $\rho$ satisfies 
\be \label{e:pdeh}
\rho_t + \cL \rho=0, \quad \rho(1, \cdot,z)=f(z,\cdot),
\ee
where  
\be \label{e:L}
\cL=\half \sigma^2 \frac{\partial ^2}{\partial y^2}+ b \frac{\partial}{\partial y}.
\ee
Using \eqref{e:pdeh}, we arrive at
\be \label{e:pdelnh}
\frac{\partial}{\partial t}\log \rho +\cL\log \rho=-\half \sigma^2 \left(\frac {\partial \log\rho}{\partial y}\right)^2.
\ee

Now, {for any $P(\alpha)\in \bar\cP(\mu_1)$, let us define the $(P(\alpha),\bar{\bbG})$-local martingale} $M=(M_t)_{t\in[0,1]}$ by
\[
M_t:=\int_0^{t\wedge \zeta(Y)}\frac{\sigma(s,X^2_s)
\rho_y(s,X^2_s;X^1_1)}{\rho(s,X^2_s;X^1_1)}dW_s,
\]
where $W$ is a {$(P(\alpha),\bar{\bbG})$-}Brownian motion.
It\^o's formula applied to $Y$ under $P(\alpha)$ then yields
\begin{align*}
\log \rho(1,X^2_1;X^1_1)&- \log \rho(0,X^2_0;X^1_1)=\lim_{T\nearrow 1} \log \rho(T,X^2_T;X^1_1)\\
&=\lim_{T\nearrow 1}M_{T} + \int_0^{T\wedge \zeta(Y)}\alpha_t\frac{\rho_y(t,X^2_t;X^1_1)}{\rho(t,X^2_t;X^1_1)}dt-\half\int_0^{T\wedge \zeta(Y)}\frac{\sigma^2(t,X^2_t)\rho^2_y(t,X^2_t;X^1_1)}{\rho^2(t,X^2_t;X^1_1)}dt\\
&=M_{1} + \int_0^{1\wedge \zeta(Y)}\alpha_t\frac{\rho_y(t,X^2_t;X^1_1)}{\rho(t,X^2_t;X^1_1)}dt-\half\int_0^{1\wedge \zeta(Y)}\frac{\sigma^2(t,X^2_t)\rho^2_y(t,X^2_t;X^1_1)}{\rho^2(t,X^2_t;X^1_1)}dt.
\end{align*}
Let us assume for the moment that $M$ is a  true $(P(\alpha),\bar{\bbG})$-martingale. Then 
\begin{multline} \label{e:alt_obj}
	\bbE^{P(\alpha)}\bigg[\half\int_0^{1\wedge \zeta(Y)}\frac{\alpha_t^2}{\sigma^2(t,X^2_t)}dt-\log \rho(1,X^2_1;X^1_1)+ \log \rho(0,X^2_0;X^1_1)\bigg]\\=\bbE^{P(\alpha)}\bigg[\half\int_0^{1\wedge \zeta(Y)}\left(\frac{\alpha_t}{\sigma(t,X^2_t)}-\frac{\sigma(t,X^2_t)\rho_y(t,X^2_t;X^1_1)}{\rho(t,X^2_t;X^1_1)}\right)^2dt\bigg].
\end{multline}
Thus, the proof will be complete as soon as we show that $M$ is a {true $(P(\alpha),\bar{\bbG})$-}martingale.

To this end, observe that
\begin{align*}
\bbE^{P(\alpha)}\bigg[\half\int_0^{1\wedge \zeta(Y)}\left(\frac{\alpha_t}{\sigma(t,X^2_t)}-\frac{\sigma(t,X^2_t)\rho_y(t,X^2_t;X^1_1)}{\rho(t,X^2_t;X^1_1)}\right)^2dt\bigg]=
H(P(\alpha)|P({\hat\alpha}))\\=-\int \log \frac{dP({\hat\alpha})}{dP(\alpha)}dP(\alpha)
=-\int \log \frac{dP({\hat\alpha})}{dR}dP(\alpha)-\int \log \frac{dR}{dP(\alpha)}dP(\alpha)\\=-\bbE^{P(\alpha)}[\log \rho(1,X^2_1;X^1_1)-\log \rho(0,X^2_0;X^1_1)]+H(P(\alpha)|R)<\infty,
\end{align*}
where the last inequality follows from the fact that $(X^1_1,X^2_1)\sim \mu_1$ under any {$P(\alpha)\in \bar\cP(\mu_1)$}. Thus the desired martingale property follows from the It\^o isometry, the triangle inequality and, once again, the fact that $H(P(\alpha)|R)<\infty$.
\end{proof}

\noindent {\bf Relation to equilibrium in the Kyle-Back model.} Recall the Kyle model with transaction costs considered in the Introduction and the description of its equilibrium via the equations \eqref{e:eqdempen}-\eqref{e:Sch1_intro}. In Section~\ref{sect.sch}, we noticed that the equilibrium strategy for the insider in this model can be expressed in terms of the Kantorovic potentials of the
entropic transport problem with both marginals equal to $\eta=\mathcal{N}(0,1)$ and quadratic cost $c_2(x,y)=\frac12(F(x)-y)^2$.

Now, for $b\equiv 0$, $\sigma\equiv 1$, and $\ell=-\infty$ (i.e. no killing),  let us choose the
target measure $\mu_1$ of the constrained \sch problem to be the solution $\pi^\eps:=\pi^\eps(\eta,\eta,c_2)$ of such entropic transport problem.
Then, it is easy to see that   the optimiser of $\bar{\mathcal{S}}(\mu_0,\pi^\eps)$ coincides with the equilibrium insider strategy in the Kyle model with transaction costs; that is, for  $\mu_1=\pi^\eps$, $\hat\alpha$ in Theorem~\ref{thm_cons} equals the insider's strategy that constitutes the drift of equilibrium total order process from \eqref{e:eqdempen}. Indeed, in view of \eqref{eq:denspieps}, the density function $f^\eps$ of $\pi^\eps$ w.r.t. $\eta\otimes\eta$ satisfies $f^\eps(z,y)=\exp{\frac{\phi_F^\eps(z)+\zeta^\eps(y)+F(z)y}{\eps}}$, with $(\phi_F^\eps,\zeta^\eps)$ related to the \sch potentials of the aforementioned entropic optimal transport problem as in \eqref{eq:relphipsi}. Moreover, it satisfies the growth assumption of the preceding theorem in view of Proposition~A.2 in \cite{KpenC}.

\section{Equivalence between constrained and unconstrained bridges}\label{sect.equiv}
In this section, we provide an alternative construction of the solution to the unconstrained \sch problem 
\eqref{P:unconstrainedeps}. While in Section~\ref{sect:unconstr} we used a unique $h$-transform, here we use two separate $h$-transforms, first for $Z$ and then for $Y$,  
relying on a filtration enlargement result that we postpone to Section~\ref{sect:enlarg} to ease the readability of this section. This will also allow us to show the equivalence between this problem and the constrained \sch problem \eqref{e:minentadpt} when the first component of $\mu_1$ is $\eta$.

\subsection{Alternative construction for the unconstrained \sch problem}
We start by noting that, under $\mu_1$, the conditional distribution of $X^2_1$ given $X_1^1$ is given by 
\be \label{x2condeps}
\mu_1(X^2_1\in dy|X^1_1=z)= \frac{f(z,y)}{f_1(z)}\eta(dy), \; y\in E_\ell.
\ee
Next consider the following function,  and observe that this is the function $\rho$ defined in \eqref{eq:rho} 
when $f_1 \equiv 1$, which is the case in the setting of the constrained \sch problem:
\be \label{e:rhodef}
\rho(t,y,z)= \int_{E}\frac{f(z,u)}{f_1(z)}q(t,1;y,u)du + \frac{f(z,\ell)}{f_1(z)}L(t,1;y), \; y\in E_\ell.
\ee

Note that $\rho(0,z_0,z)\equiv 1$ by the definition of $f_1$, and that $\rho(t,\ell,z)= \frac{f(z,\ell)}{f_1(z)}$ since $L(t,1;\ell)=1$.
For each $z$, the function $\rho(\cdot,\cdot,z)$ can be considered as an $h$-function to condition the distribution of $X^2_1$ to be given by \eqref{x2condeps} whenever $X^1_1=z$. 
On the other hand, the $\mu_1$-distribution of $X^1_1$ has density $f_1$ with respect to $\eta$. Thus, one needs to first condition $Z$ from  \eqref{eq:Zref} so that $Z_1$ has such distribution. To this end, consider the map
\[
(t,z)\mapsto\int_{E}f_1(u)q(t,1;z,u)du + f_1(\ell)L(t,1;z),
\]
and note that it coincides with $\pi^2$ defined in Corollary \ref{c:conddist}. 
In the next two lemmas, we successively perform an $h$-transform for $Z$ and for $Y$. In fact, using $\pi^2$ as an $h$-function and using Theorem 2.2 in \cite{CDbridge}, we obtain:
\begin{lemma}
Suppose Assumption \ref{a:main} holds. There exists a unique weak solution to 
\be \label{e:Zhtr}
dZ_t= \chf_{Z_t\in E}\left\{b(t,Z_t) + \sigma^2(t,Y_t)\frac{\pi^2_x(t,Z_t)}{\pi^2(t,Z_t)}\right\}dt + \chf_{Z_t\in E}\sigma(t,Z_t)d\beta_t
\ee
on some probability space $(\Om, \cG, \bbQ)$   supporting a Brownian motion $\beta$.  Moreover, $Z$ is strong Markov with transition dynamics
\be \label{e:trdyn}
\bbQ(Z_t\in dw|Z_s=z)=q(s,t;z,w)\frac{\pi^2(t,w)}{\pi^2(s,z)}dw\chf_{w\in E}+ L(s,t;z)\frac{\pi^2(t,\ell)}{\pi^2(s,z)}\delta_{\ell}(dw).
\ee  
\end{lemma}

Next, we initially enlarge the natural filtration of $Z$ from \eqref{e:Zhtr} with $Z_1$, thus arriving at the following result. 
\begin{lemma}\label{lemma:tilde}
Suppose Assumption \ref{a:main} holds and consider a filtered probability space $(\Om,\cG, \bbG, \bbQ)$, with $\bbG=(\cG_t)_{t\in [0,1]}$ satisfying the usual conditions, which supports a 2-dimensional standard Brownian motion $(\beta_1,\beta_2)$  as well as a $\cG_0$-measurable random variable $Z_1$ with $\bbQ(Z_1\in dz)=f_1(z)\eta(dz)$.  Then there exits a unique  $\bbG$-adapted  strong solution to 
\begin{equation}\label{eqXtilde}
\begin{split}
    dZ_t&= \chf_{Z_t\in E}\left\{b(t,Z_t) + \sigma^2(t,Z_t)\bigg(\frac{q_y(t,1;Z_t,Z_{1})}{q(t,1;Z_t,Z_{1})}\chf_{Z_{1}\in E}+ \frac{L_z(t,1;Z_t)}{L(t,1;Z_t)}\chf_{Z_{1}=\ell}\bigg)\right\}dt \\
&+\chf_{Z_t\in E}\sigma(t,Z_t)d\beta^1_t, \quad t<1,\quad Z_0=z_0;  \\
dY_t &=\chf_{Y_t\in E}\left\{b(t,Y_t) + \sigma^2(t,Y_t)\frac{\rho_y(t,Y_t,Z_{1})}{\rho(t,Y_t,Z_{1})}dt + \chf_{Y_t\in E}\sigma(t,Y_t)d\beta^2_t\right\},\quad t<1,\quad Y_0=z_0. 
\end{split}
\end{equation}
Moreover, $Z_t \to Z_1$ a.s., and, in its own filtration, $Y$ follows
\begin{equation}\label{eq.Yfil}
dY_t =\chf_{Y_t\in E}\left\{\Big(b(t,Y_t)+\sigma^2(t,Y_t)\frac{\pi^1_y(t,Y_t)}{\pi^1(t,Y_t)}\Big) dt +\sigma(t,Y_t)dB^Y_t\right\},
\end{equation}
where $B^Y$ is a Brownian motion adapted to the natural filtration of $Y$.
\end{lemma} 
Note that the above decomposition of $Z$ is its Doob-Meyer decomposition when its filtration is initially enlarged with $Z_1$. 
\begin{proof}
In view of \eqref{e:Zhtr} and \eqref{e:trdyn}, it follows from Theorem 2.2 and Remark 2.1 in \cite{CDbridge} that there exits a unique weak solution  with $Z_1=\lim_{t\rar 1}Z_t$. Moreover, the drift and volatility coefficients are locally Lipschitz implying, by Theorem 2.2.8 in \cite{DMB-CD}, pathwise uniqueness until the first exit time from $E$. Thus, by continuity at the exit time, there exists a unique strong solution by Yamada-Watanabe theorem (see, e.g., Theorem 2.2.12 in \cite{DMB-CD}).

Next observe that, conditional on $Z_1=z$, the law of $Y$ given by \eqref{eqXtilde} coincides with that of the $h$-transform of $X^2$ under $R$ via the $h$-function $\rho(\cdot, \cdot, z)$. Thus, using that $\rho(0,z_0,z)=1$,
for $y_i\in E$ for all $i=1,\ldots, n$, we have
\begin{equation}\begin{split}\label{eq.y1nz1}
\bbQ(Y_{t_1}\in dy_1,  \ldots, Y_{t_n}\in dy_n|Z_{1}=z)=\hspace*{4cm}\\
q(0,t_1;z_0,y_1)q(t_1,t_2;y_1,y_2)\ldots q(t_{n-1},t_n;y_{n-1},y_n)\rho(t_n,y_n,z)dy_1\ldots dy_n.
\end{split}
\end{equation}
Similarly, for $y_i\in E$ for all $i=1,\ldots, n-1$, and $y_n=\ell$,
\begin{eqnarray}\label{eq.y1nlz1}\begin{split}
\bbQ(Y_{t_1}\in dy_1,  \ldots, Y_{t_{n-1}}\in dy_{n-1}, Y_{t_n}=\ell|Z_{1}=z)=\hspace*{4cm}\\
q(0,t_1;z_0,y_1)q(t_1,t_2;y_1,y_2)\ldots L(t_{n-1},t_n;y_{n-1})\rho(t_n,\ell,z)dy_1\ldots dy_n.
\end{split}
\end{eqnarray}
Consequently, Bayes' formula yields, for $z\in E_\ell$,
\begin{align*}
  \bbQ(Z_{1}\in dz|Y_0=y_0, Y_{t_1}=y_1, \ldots, Y_{t_n}=y_n)&\propto \bbQ(Y_0\in dy_0,  \ldots, Y_{t_n}\in dy_n|Z_1=z)\bbQ(Z_{1}\in dz)\\  
  &\propto \rho(t_n,y_n,z)f_1(z)\eta(dz).
\end{align*}
This in turn implies
\be \label{e:Z1deltdist} 
\bbQ(Z_{1}\in dz|\cF^Y_t)= \frac{\rho(t,Y_t,z)f_1(z)\eta(dz)}{\int_{E_\ell} \rho(t,Y_t,w)f_1(w)\eta(dw)}
=\frac{\rho(t,Y_t,z)f_1(z)\eta(dz)}{\pi^1(t,Y_t)}.
\ee 
Thus, in its own filtration $Y$ has a drift given by $b(t,Y_t)+\sigma^2(t,Y_t)m(t,Y_t)$, where
\[
m(t,y)=\frac{\int_{E_\ell} \rho_y(t,Y_t,w)f_1(w)\eta(dw)}{\pi^1(t,y)}=\frac{\pi^1_y(t,y)}{\pi^1(t,y)}.
\]
\end{proof}

\subsection{Constrained vs unconstrained problem}
Note that \eqref{eq.Yfil} is the same SDE as the one in \eqref{e:Schd2NF}. In the next theorem we will show that the distribution of $(Z,Y)$ in Lemma~\ref{lemma:tilde} (unique solution to the constrained problem) coincides with $P^*$ from Theorem \ref{thm:first}, i.e. the unique minimiser of \eqref{P:unconstrainedeps} (unique solution to the unconstrained problem).
\begin{theorem}\label{thm.cunc}
 Let $(Z,Y)$ be as in  \eqref{eqXtilde}, and denote by $\bbH=(\cH_t)_{t\in [0,1]}$ the minimal filtration satisfying the usual conditions with respect to which $(Z,Y)$ is adapted. Then
\begin{align}\label{eq.Z}
dZ_t&=\chf_{Z_t\in E}\left\{\sigma(t,Z_t)d\beta_t+
(b(t,Z_t) + \sigma^2(t,Z_t)\frac{h_z(t,(Z_t,Y_t))}{h(t,(Z_t,Y_t))})dt\right\},\\
dY_t&=\chf_{Y_t\in E}\left\{\sigma(t,Y_t)d\tilde \beta_t+
(b(t,Y_t) + \sigma^2(t,Y_t)\frac{h_y(t,(Z_t,Y_t))}{h(t,Z_t,Y_t)})dt\right\},\label{eq.Y}
\end{align}
where $(\beta,\tilde \beta)$
is a 2-dimensional $\bbH$-Brownian motion.
\end{theorem}

\begin{proof}The stated decomposition will follow from Theorem \ref{theorem:H} below. Therefore, we are going to show that Assumptions~\ref{ass:mk}, \ref{ass:cond} and \ref{ass:rho} are satisfied in this setting. Throughout this proof, we denote by $\bbE$ the expectation operator with respect to $\bbQ$.

To see that Assumption~\ref{ass:mk} holds, that is, that $Z$ and $Y$ are jointly Markov, let $g$ and $j$ be continuous and bounded and consider for $s<t\leq 1$
\bean
\bbE[g(Z_t)j(Y_t)|\cH_s]&=&\bbE[\bbE[g(Z_t)j(Y_t)|\cH_s, Z_1]|\cH_s]\\
&=&\bbE[\bbE[g(Z_t)|\cF^Z_s, Z_1]\bbE[j(Y_t)|\cF_s^{Y}, Z_1]|\cH_s]\\
 &=&\bbE[\bbE[g(Z_t)|Z_s, Z_1]\bbE[j(Y_t)|Y_s, Z_1]|\cH_s]\\
&=&\bbE[\phi(s,t,Z_s,Z_1)\psi(s,t,Y_s,Z_1)|\cH_s],
\eean
where 
\[   \phi(s,t,Z_s,Z_1):= \bbE[g(Z_t)|Z_s, Z_1] \quad \mbox{ and } \quad \psi(s,t,Y_s,Z_1):=\bbE[j(Y_t)|Y_s, Z_1].
\]
In the chain of equalities above, we used the independence of $Z$ and $Y$ given $Z_1$, that $Z$ is a Markov process, and that $Y$ given $Z_1$ is Markov as well; see Lemma~\ref{lemma:tilde}. Moreover, since $Z$ is Markov and \eqref{e:Z1deltdist} holds, $Z_1$ is independent from $\cH_s$ given $Z_s$ and $Y_s$. This in turn implies
\[
\bbE[g(Z_t)j(Y_t)|\cH_s]=\bbE[\phi(s,t,Z_s,Z_1)\psi(s,t,Y_s,Z_1)|Z_s, Y_s].
\]

In order to verify that Assumptions \ref{ass:cond} and \ref{ass:rho} hold, we shall first compute $\bbQ(Y_s\in dy, Y_t \in dz|\cF^Z_t)$ on the event $\{Z_t\in E\}$. Recall that the transition probabilities for $Z$ under $\bbQ$ are given by \eqref{e:trdyn}. Then,
\bean
&&\bbQ(Y_s\in dy, Y_t \in dz|\cF^Z_t)=\bbE\left[\bbQ(Y_s\in dy, Y_t \in
  dz|\cF^Z_t, Z_1)|\cF^Z_t\right] =\bbE\left[\bbQ(Y_s\in dy, Y_t \in
  dz| Z_1)|\cF^Z_t\right] \\
&&=\chf_{(y,z)\in E^2}P^o(s,t,y,z)\left(\int_E \rho(t,z,w)\frac{q(t,1;Z_t,w)f_1(w)}{\pi^2(t,Z_t)}\,dw +\frac{L(t,1;Z_t)f_1(\ell)\rho(t,z,\ell)}{\pi^2(t,Z_t)} \right) dy dz\\
&&+\chf_{y\in E}\chf_{z=\ell}
P^o(s,t,y,z)\left(\int_E \rho(t,z,w)\frac{q(t,1;Z_t,w)f_1(w)}{\pi^2(t,Z_t)}\,dw +\frac{L(t,1;Z_t)f_1(\ell)\rho(t,z,\ell)}{\pi^2(t,Z_t)} \right) dy\delta_{\ell}(dz)\\
&&+\chf_{y=z=\ell}P^o(s,t,y,z)\left(\int_E \rho(t,\ell,w)\frac{q(t,1;Z_t,w)f_1(w)}{\pi^2(t,Z_t)}\,dw +\frac{L(t,1;Z_t)f_1(\ell)}{\pi^2(t,Z_t)} \rho(t,\ell,\ell)\right)\delta_{\ell}(dy)\delta_{\ell}(dz),
\eean
where
\be \label{e:Porep}
\begin{split}
P^o(s,t,y,z)&=\chf_{(y,z)\in E^2}q(0,s;z_0,y)q(s,t;y,z)\\
&+\chf_{y\in E}\chf_{z=\ell}
q(0,s;z_0,y)L(s,t;y)+\chf_{y=z=\ell}L(0,s;z_0).
\end{split} 
\ee
This shows that \eqref{ass:cond_1} holds.

Moreover, by choosing $m(dy)=\chf_E(y)dy +\delta_{\ell}(dy)$, we see that 
\[
\bbQ(Y_s \in dy, Y_t \in dz |Z_t=u)=\Phi(s,t,u,y,z)\,m(dy)\,m(dz)\quad\text{for all}\, s,t\geq 0,\ y,z,u\in E_\ell,
\]
where
\be \label{e:Prepresent}
\Phi(s,t,u,y,z)=\frac{P^o(s,t,y,z)}{\pi^2(t,u)}\left(\int_E \rho(t,z,w)q(t,1;u,w)f_1(w)\,dw +L(t,1;u)f_1(\ell)\rho(t,z,\ell)\right),
\ee
so Assumption~\ref{ass:cond} is satisfied.

Next, we are going to show that Assumption~\ref{ass:rho} is satisfied.  Observe that, by means of Chapman-Kolmogorov identity,
\begin{align*}
G(y,w):=\int_{E_\ell} P^o(s,t,y,z) \rho(t,z,w)  f_1(w)m(dz)&=\chf_{y\in E}q(0,s;z_0,y)\int_E f(w,x)q(s,1;y,x)dx\\
&+ \chf_{y\in E}q(0,s;z_0,y)f(w,\ell)L(s,1;y)\\
&+ \chf_{y=\ell}f(w,\ell)L(0,s;z_0),
\end{align*}
and note that the above does not depend on $t$. Thus, Assumption~\ref{ass:rho}-(i) will be satisfied as soon as, for each $y$, 
\[
(t,u)\mapsto \int_E G(y,w)q(t,1;u,w)\,dw + G(y,\ell) L(t,1;u)
\]
belongs to $C^{1,2}([0,1)\times E,\bbR)$. To this end, first observe that, for $\alpha < \frac{C}{4(1-s)}$, where  $C$ is the constant in Proposition \ref{prop.reg}-(2), there exists a constant $K$ such that $f(w,x)\leq K\exp(\alpha (x+w)^2)$. 
Thus,
\[
\int_E f(w,x)q(s,1;y,x)dx\leq K \int_E \exp(\alpha (x+w)^2q(s,1;y,x)dx \leq K_0(y)\exp(2 \alpha w^2),
\]
where $K_0(y)$ is a finite constant depending continuously on $y$. Therefore, the desired smoothness follows from the dominated convergence theorem, in view of the bounds on $q$ and its derivatives given in Proposition~\ref{prop.reg}. Similar arguments show that  Assumption~\ref{ass:rho}-(ii) holds, too.

Also observe that the $u$-derivative of $\frac{\Phi}{P^o}$ is locally bounded due to the Gaussian bounds on the transition density $q$ (see Proposition \ref{prop.reg}-(\ref{qintUB})) and since  $\pi^2$ is strictly positive. Moreover, $P^o(s,t,\cdot,\cdot) \in L^1(E_\ell^2, dm\otimes dm)$ by direct calculations using Chapman-Kolmogorov identity. 
Therefore, Assumption~\ref{ass:rho}-(iii) holds as well.

Next, using the definition of $\rho$, one obtains
\[
\int_E \rho(t,z,w)q(t,1;u,w)f_1(w)\,dw +L(t,1;u)f_1(\ell)\rho(t,z,\ell)=h(t,(u,z)).
\]
Thus, \eqref{e:Prepresent} implies
\[
\Phi(s,t,u,y,z)=P^o(s,t,y,z)\frac{h(t,(u,z))}{\pi^2(t,u)}.
\]

As all conditions are satisfied, we can apply Theorem~\ref{theorem:H}, that yields the desired decomposition \eqref{eq.Z} of $Z$ under the enlarged filtration $\bbH$. 

Similar calculations establish the result for $Y$. We compute the relevant conditional distribution when there is no killing and leave the remaining details to the reader. To determine the extra drift that $Y$ receives  when enlarging its filtration progressively with $Z$, we compute, using \eqref{e:Z1deltdist},
\begin{align*}
\bbQ(Z_s\in dx, Z_t\in dz|&Y_t=y)=\bbE[\bbQ(Z_s\in dx, Z_t\in dz|Z_{1})|Y_t=y] \\
&=\frac{q(0,s; z_0,x)q(s,t;x,z)}{\pi^1(t,y)}\int_E \frac{\rho(t,y,w)f_1(w)q(0,1;z_0,w)q(t,1;z,w)}{q(0,1;z_0,w)}dw\\
&=\frac{q(0,s; z_0,x)q(s,t;x,z)}{\pi^1(t,y)}h(t,(z,y)).
\end{align*}
\end{proof}

\section{Limiting \sch bridges}\label{sect.lim}
The aim of this section is to relate the equilibrium problem in the classical Kyle model (without penalties) to some \sch bridge problem. We will see that this needs to be done via approximation, by introducing auxiliary equilibrium problems with penalty, and then letting the penalty go to zero.
Indeed,  if we interpret the canonical process $(X^1,X^2)$ as the  private signal of the insider and the demand process, then in equilibria considered in \eqref{e:kyleB} and \eqref{e:eqdemdef}, the joint distribution of $(X^1_1,X^2_1)$ is given by
\[
\mu_1(dx_1,dx_2)=\eta(dx_1)\delta_{x_1}(dx_2),
\]
which is clearly singular with respect to $\eta\otimes\eta$; that is,  the distribution of $(X^1_1,X^2_1)$ under the reference measure $R$ set at the beginning of Section~\ref{sect:c_and_unc}.

This means that a \sch problem with reference measure $R$ and such target distribution $\mu_1$ is ill-posed. In order to circumvent this problem, we are going to consider a sequence of target distributions $(\mu_1^{\eps})_{\eps>0}$ on $E^2$ (or  $E_{\ell}^2$ if $\ell$ is finite) that are absolutely continuous with respect to $\eta\otimes\eta$ and converge weakly to $\mu_1$: 
\begin{equation}\label{eq.mueps}
 \mu^{\eps}_1\ll\eta\otimes\eta\quad \forall\; \eps>0,\quad\quad  \mu^\eps_1\rightharpoonup\mu_1,
\end{equation}
and study the associated \sch problems.
We use similar notations to \eqref{eq.dens}, and denote the Radon-Nikodym density of $\mu^\eps_1$ with respect to $R_1$  by $f^\eps$, and the Radon-Nikodym densities of its marginals with respect to $\eta$  by $f^\eps_1$ and $f^\eps_2$, respectively.
\begin{assumption}\label{a:feps} The measures $\{\mu^\eps_1\}_{\eps>0}$ satisfy the following conditions:
\begin{enumerate}
  \item  For each $\eps>0$,  $f^\eps$ satisfies Assumption~\ref{ass:f};
  \item For $\eps\to 0$, $f^\eps_1$ and $f^\eps_2$ converge to $1$ in $L^1(\eta)$;
\item The mass assigned to the atoms of $\mu_1^{\eps}$ satisfies 
\begin{align*}
    \lim_{\eps\to 0}\mu_1^{\eps}(E\times\{\ell\})+ \mu_1^{\eps}(\{\ell\}\times E)=0, \mbox{ and}\\ \lim_{\eps\rar 0}\mu^\eps_1(\{\ell\}\times\{\ell\})=\mu(\{\ell\}\times\{\ell\})=L(0,1;z_0).
\end{align*}
\end{enumerate}  
\end{assumption}

The next theorem identifies the limit of the minimisers of the Schr\"odinger problems $\mathcal{S}_E(\mu_0,\mu_1^\eps)$ as $\eps\rar 0$.
To this end, let us denote
\be \label{e:h0def}
h^0(t,(z,y)):=\int_{E}\frac{
q(t,1;z,w)q(t,1;y,w)}{q(0,1;z_0,w)}dw +\frac{L(t,1;z)L(t,1;y)}{L(0,1;z_0)}.
\ee

\begin{theorem}\label{thm:conv} Suppose Assumptions~\ref{a:main},\ref{ass:f} and  \ref{a:feps} hold. 
As $\eps\to0$, the solutions of the Schr\"odinger problems $\mathcal{S}_E(\mu_0,\mu_1^\eps)$
weakly converge to the distribution of
\begin{equation}
\label{e:epsbridge_g_lim}
\begin{split}
X_t &=x_0+\int_0^t \mathbf{I}(u, X_u)\cdot
\left\{\mathbf{b}(u,X_u) + \Sigma^2(u,X_u)\frac{(\nabla h^0(u,X_u))^*}{h^0(u,X_u)}\right\} du\\
&\hspace{1cm}+\int_0^t \mathbf{I}(u, X_u)\cdot\Sigma(u,X_u)d\beta_u,
\end{split}
	\end{equation}
	where $x_0=(z_0,z_0)$, and $\beta$ is a 2-dimensional standard Brownian motion.
\end{theorem}
Recall that, by Theorem~\ref{thm:first}, the solution to \eqref{e:epsbridge_g} with $h$ replaced by
\[
h^\eps(t,(z,y)):=\bbE^R\bigg[f^\eps(X^1_1,X^2_1) \Big|(X^1_t,X^2_t)=(z,y)\bigg]
\]
is the solution to the \sch problem $\mathcal{S}_E(\mu_0,\mu_1^\eps)$. Therefore, the convergence in Theorem~\ref{thm:conv} boils down to the weak convergence of solutions of SDEs.

\begin{proof}
Let $P^\eps$ be the solution of the Schr\"odinger problem $\mathcal{S}_E(\mu_0,\mu^\eps)$. We want to first prove tightness of the set of the measures $\{P^\eps\}_\eps$, by showing that for any $c>0$
\[
\lim_{\delta\to 0}\limsup_{\eps\to 0} P^\eps(w(X^\eps,\delta)>c)=0,
\]
where 
\[
w(X^\eps,\delta)=\sup_{s,t\in[0,1] :  |s-t|\leq \delta}\|X^\eps_s-X^\eps_t\|.
\]
First, 
\begin{align*}
&E^R\left[
1_{\{w(X,\delta)>c\}}f^{\eps}(X^1_1,X^2_1)\right]=E^R\left[E^R\left[
1_{\{w(X,\delta)>c\}}f^{\eps}(X^1_1,X^2_1)\big|X^1_1, X^2_1\right]\right]\\
&=\int_{E^2_{\ell}}R\left(w(X,\delta)>c|X^1_1=z,X^2_1=y\right)f^{\eps}(z,y)\eta(dz)\eta(dy)\\
&
\leq
\int_{E^2_{\ell}}\left(R\left(w(X^1,\delta)>\frac{c}{2}\middle|X^1_1=z\right)
+R\left(w(X^2,\delta)>\frac{c}{2}\middle|X^2_1=y\right)\right)
f^{\eps}(z,y)\eta(dz)\eta(dy)
\\
&
=\int_{E_\ell}R\left(w(X^1,\delta)>\frac{c}{2}\middle|X^1_1=z\right)f^\eps_1(z)\eta(dz) + \int_{E_\ell}R\left(w(X^2,\delta)>\frac{c}{2}\middle|X^2_1=y\right)f^\eps_2(y)\eta(dy).
\end{align*}
By the uniform integrability hypothesis in Assumption \ref{a:feps}, taking the limit for $\eps\to 0$, the above converges to 
\begin{align*}
\int_{E_\ell}R\left(w(X^1,\delta)>\frac{c}{2}\middle|X^1_1=z\right)\eta(dz) + \int_{E_\ell}R\left(w(X^2,\delta)>\frac{c}{2}\middle|X^2_1=y\right)\eta(dy)\\
=R\left(w(X^1,\delta)>\frac{c}{2}\right) +R\left(w(X^2,\delta)>\frac{c}{2}\right).
\end{align*}
This completes the proof of the tightness of $\{P^\eps\}_\eps$ since, for $i=1, 2$, $R(\{w(X^i,\delta)>c/2\})\rar 0$ as $\delta\to 0$,  due to the tightness of the law of $X^i$s under $R$.

Thus, to determine the limiting distribution, it suffices to consider the limit of 
\[
\bbE^R[\phi(X_{t_1},\ldots,X_{t_n})f^{\eps}(X_1)]=\bbE^R[\phi(X_{t_1},\ldots,X_{t_n})h^\eps(t_n, X_{t_n})]
\]
as $\eps \to 0$, for any $0\leq t_1\leq\ldots\leq t_n<1$ and continuous and bounded function $\phi$. We shall show that the above limit is given by
\be \label{e:limh0}
\bbE^R[\phi(X_{t_1},\ldots,X_{t_n})h^0(t_n, X_{t_n})].
\ee
To do so, first observe that 
\[
g_\phi (x)= \bbE^R[\phi(X_{t_1},\ldots,X_{t_n})|X_1=x]
\]
is bounded and continuous on $E^2$. Indeed, denoting the transition function of $X$ under $R$ on $E_\ell^2$ by $(R_{s,t}(\cdot, \cdot))_{0\leq s\leq t\leq 1}$, we obtain for $x\in E^2$ that 
\[
g_\phi (x)=\frac{\int_{E_{\ell}\times \ldots E_{\ell}}\phi(x_1,\ldots,x_n)R_{0,t_1}((z_0,z_0),dx_1)R_{t_1,t_2}(x_1,dx_2)\ldots R_{t_{n-1},t_n}(x_{n-1},dx_n)p(t_n,1;x_n,x)}{p(0,1;(z_0,z_0),x)}.
\]
The desired continuity then follows from the continuity of the transition density $q$ and the boundedness of $\phi$. 

Similarly, for $x\in \del E_{\ell}^2$, 
\[
g_\phi (x)=\frac{\int_{E_{\ell}\times \ldots E_{\ell}}\phi(x_1,\ldots,x_n)R_{0,t_1}((z_0,z_0),dx_1)R_{t_1,t_2}(x_1,dx_2)\ldots R_{t_{n-1},t_n}(x_{n-1},dx_n)J(t_n,1;x_n)}{J(0,1;(z_0,z_0))},
\]
where $J(t,1;x):=R(X_1\in \del E_{\ell}^2|X_t=x)$. Thus,
\begin{align}
    \lim_{\eps\rar 0} &\bbE^R[\phi(X_{t_1},\ldots,X_{t_n})f^{\eps}(X_1)]=\lim_{\eps \to 0}\int_{E_{\ell}^2}g_\phi(x)f^{\eps}(x)\eta\otimes\eta(dx)=\lim_{\eps \to 0}\int_{E_{\ell}^2}g_\phi(x)\mu_1^\eps(dx)\nn \\
    &=\lim_{\eps \to 0}\mu_1^{\eps}(E^2)\int_{E^2}g_\phi(x)\frac{\mu_1^\eps(dx)}{\mu_1^\eps(E^2)}+\lim_{\eps \to 0}\int_{\del E_{\ell}^2}g_\phi(x)\mu_1^\eps(dx)\nn \\
    &=\int_E g_{\phi}(z,z)q(0,1;z_0,z)dz+g_\phi((\ell,\ell))L(0,1;z_0),\label{e:limfdd}
\end{align}
where the last equality follows from the assumption that $\mu_1^{\eps}(E\times\{\ell\})+ \mu_1^{\eps}(\{\ell\}\times E)\to 0$ as $\eps \to 0$, and that $\lim_{\eps\rar 0}\mu_1^\eps(\{\ell\}\times\{\ell\})=\mu_1(\{\ell\}\times\{\ell\})=L(0,1;z_0)$. Also note that $\mu_1^\eps/\mu_1^\eps(E^2)$ is a probability measure on $E$ converging weakly to $\mu_1/\mu_1(E^2)$ under Assumption \ref{a:feps}.

On the other hand, using the definition of $g_\phi$,
\[
\int_E g_{\phi}(z,z)q(0,1;z_0,z)dz=\bbE^R\left[\phi(X_{t_1},\ldots,X_{t_n})\int_E\frac{q(t_n,1;X^1_{t_n},z)q(t_n,1;X^2_{t_n},z)}{q(0,1;z_0,z)}dz\right].
\]
Moreover, 
\[
g_\phi((\ell,\ell))= \bbE^R\left[\phi(X_{t_1},\ldots,X_{t_n})\frac{L(t_n,1;X^1_{t_n})L(t_n,1;X^2_{t_n})}{L^2(0,1;z_0)}\right].
\]
Consequently,  $\bbE^R[\phi(X_{t_1},\ldots,X_{t_n})h^0(t_n, X_{t_n})]$ coincides with \eqref{e:limfdd}, which completes the proof. 
\end{proof}

\begin{remark} \label{r:Kyle} Note that, from Lemma~\ref{lemma:tilde} applied to $\mu_1^\eps$, the joint distribution of the pair of processes in \eqref{eqXtilde}, with $\rho$ replaced by 
\begin{equation*}
\rho^\eps(t,y,z):= \int_{E}\frac{f^\eps(z,u)}{f^\eps_1(z)}q(t,1;y,u)du + \frac{f^\eps(z,\ell)}{f^\eps_1(z)}L(t,1;y), \; y\in E_\ell,
\end{equation*}
converges as $\eps \to 0$  to the distribution of $(Z,Y)$, where
\be \label{e:Kyle2} \begin{split}
dZ_t &=\chf_{Z_t\in E}\bigg\{b(t,Z_t) + \sigma^2(t,Z_t)\bigg(\frac{q_z(t,1;Z_t,Z_{1})}{q(t,1;Z_t,Z_{1})}\chf_{Z_{1}\in E}+ \frac{L_z(t,1;Z_t)}{L(t,1;Z_t)}\chf_{Z_{1}=\ell}\bigg)\bigg\}dt \\
&+\chf_{Z_t\in E}\sigma(t,Z_t)d\beta^1_t, \quad Z_0=z_0,\\
    dY_t &=\chf_{Y_t\in E}\bigg\{b(t,Y_t) + \sigma^2(t,Y_t)\bigg(\frac{q_z(t,1;Y_t,Z_{1})}{q(t,1;Y_t,Z_{1})}\chf_{Z_{1}\in E}+ \frac{L_z(t,1;Y_t)}{L(t,1;Y_t)}\chf_{Z_{1}=\ell}\bigg)\bigg\}dt \\
&+ \chf_{Y_t\in E}\sigma(t,Y_t)d\beta^2_t, \quad Y_0=z_0,
\end{split}
  \ee  
$(\beta^1,\beta^2)$ is a 2-dimensional $\bar{\bbG}$-Brownian motion, and $\bar{\bbG}$ is the minimal filtration satisfying the usual conditions generated by $(\beta^2,Z)$ and $Z_1$. 
\end{remark}

The above weak convergence result in particular allows us to understand the sensitivity of the value of private information on transaction costs in  the Kyle-Back model,  which is defined to be  the reservation value for a risk-neutral uninformed trader of knowing the exact value of $V$. Note that without any private information on $V$, the expected terminal wealth of any strategic trader is $0$ since prices remain constant and equal to $\bbE[V]$ all the time due to he demand for the asset being independent of $V$ in the absence of insider trading. Thus, while the insider's expected profit is calculated conditional on the realisation of her private signal $V=v$, the value of information equals the unconditional expectation of the terminal profit. This unconditional expectation amounts to the  the maximum value that a risk-neutral uninformed trader would pay to receive this private information.
\begin{corollary}
    \label{cor.convinsprofit}  Consider the Kyle model with quadratic penalties from Section \ref{s:kyle} parametrised by $\eps>0$. As $\eps\rar 0$, the value of private information converges to that in the classical Kyle model without quadratic transaction costs.
    \end{corollary}
\begin{proof}
     Theorem~5.1 in \cite{KpenC} shows that, given the realization of $V=v$, the insider's expected profit from trading equals
      \[
   -\phi^\eps_F(F^{-1}(v))-\int\zeta^\eps(y)\gamma(dy),
   \]
    where $\gamma$ is the standard normal distribution and the potentials $(\phi_F^\eps,\zeta^\eps)$ are solving \eqref{e:Sch1_alt}.
    Thus, the value of this private information for an uninformed strategic trader equals
  \begin{equation}\label{eq.profits}
   -\int\phi_F^\eps(z)\gamma(dz)-\int\zeta^\eps(y)\gamma(dy).
   \end{equation}

   Note that, for $\eps=0$, the non-regularized problem $OT(\gamma,\gamma,c_2)$ has a unique solution (of Monge form, with optimal map given by $T=F^{-1}$), and its dual problem has a unique solution too, say $(\varphi,\psi)$. Therefore, by \cite[Theorem 5.16]{Marcelnotes}, for $\eps\to 0$, the unique solution $(\varphi^\eps, \psi^\eps)$  to the regularized transport problem $\text{EOT}_\eps(\gamma,\gamma,c_2)$ satisfies $\varphi^\eps\to \varphi$ and $\psi^\eps\to\psi$ in $L^1(\gamma)$. Recalling the relation \eqref{eq:relphipsi} between $(\phi_F^\eps,\zeta^\eps)$ and $(\varphi^\eps, \psi^\eps)$,
the insider's expected profit \eqref{eq.profits} in the quadratic penalties case converges to 
   \[
    -\int\varphi(z)\gamma(dz)+\int \frac{F(z)^2}{2}\gamma(dz)-\int\psi(y)\gamma(dy)+\int \frac{y^2}{2}\gamma(dy)=\int \frac{F(y)^2+y^2}{2}\gamma(dy),
   \]
 which is the insider's ex-ante profit in the case without penalties, see Theorem 5.1 in \cite{CD-GKB}.
 Indeed, the fact that $\int\varphi+\psi\  d\gamma=0$ follows from duality and the fact that the optimal cost for the problem $OT(\gamma,\gamma,c_2)$ is zero.
\end{proof}

\begin{example} 
Consider the setting of Corollary \ref{cor.convinsprofit} and further assume that $V$ is standard normal, i.e. $F=Id$.
 In this case, Example~4.2 in \cite{KpenC} shows that
   \[
   \phi_F^\eps(z)=-\frac{\eps}{2}\log\frac{\eps}{\eps+\lambda}-\frac{z^2}{2(\eps+\lambda)}\mbox{ and } \zeta^\eps(y)=-\frac{\lambda y^2}{2},
   \]
   where
   \[
   \lambda = \frac{-\eps+\sqrt{\eps^2 + 4}}{2}.
   \]
   In view of the discussion at the end of Section~\ref{sect:c_and_unc}, if we set $\mu_1^\eps=\pi^\eps$, where $\pi^\eps$ is the optimal coupling of the associated entropy regularized optimal transport problem, the optimizer of $\bar{\mathcal{S}}(\mu_0,\mu_1^\eps)$ (see  \eqref{e:minentadpt})  is given by
   \[
\hat\alpha_t=\frac{\rho^\eps_y(t,X^2_t;X^1_1)}{\rho^\eps(t,X^2_t;X^1_1)},
   \]
   where 
   \[
   \rho^\eps(t,y,z)=\int \exp\Big(\frac{uz+\phi_F^\eps(z)+\zeta^\eps(u)}{\eps}\Big)\frac{1}{\sqrt{2\pi(1-t)}}\exp\Big(-\frac{(y-u)^2}{2(1-t)}\Big)du.
   \]
   Then, straightforward calculations show that
   \[
   \hat\alpha_t=\lambda \frac{X^1_1-\lambda X^2_t}{1-t\lambda^2},
   \]
   which in particular agrees with the optimal insider strategy given in Corollary 5.3 in \cite{KpenC}. Moreover, it converges to the Brownian bridge strategy of the classical Kyle model (see \eqref{e:kyleB}) when $\eps\to 0$.
   \end{example}

\section{Progressive enlargement with processes}\label{sect:enlarg}
Let $W$ be a standard Brownian motion on a filtered probability space $(\Om, \cF, \bbF, \bbP)$ and consider the diffusion process $Z$ defined by \eqref{eq:Zref}.  We consider enlarging the filtration $\bbF=(\cF^Z_t)_{t\in [0,1]}$ with $\bbG$, where $\bbF$ is the minimal filtration satisfying the usual conditions with respect to which $Z$ is adapted, and $\bbG$ is the natural filtration of another process $Y$. 

We are interested in the $\bbH$-canonical decomposition of $Z$ where
$\bbH=(\cH_t)_{t\in [0,1]}$ and $\cH_t=\sigma(\cF^Z_s, \cG_s, s\leq t)$. We denote the state space of $Y$ by $J$ and  assume the following
Markov structure.
\begin{assumption}\label{ass:mk} The pair $(Z,Y)$ is Markov with respect to $\bbH$.
\end{assumption}
The following assumption on the conditional distribution of $Y$ will be needed.
\begin{assumption}\label{ass:cond}  For any continuous bounded function $g:J\mapsto \bbR$ and $s \leq t$,
\begin{equation}\label{ass:cond_1}
\bbE[g(Y_s, Y_t)|\cF^Z_t]=\bbE[g(Y_s,Y_t)|Z_t].
\end{equation}
Moreover, there exist a Borel measure $m$ on $J$ and a function $\Phi$ on $[0,1]^2\times E_\ell\times J^2$ such that $\Phi(s,t,u,.,.)$ is strictly positive for all $s,t\geq 0, u\in E_\ell\ m\otimes m$-a.e., and
\[
\bbP(Y_s \in dy, Y_t \in dz |Z_t=u)=\Phi(s,t,u,y,z)\,m(dy)\,m(dz)\quad\text{for all } s,t\geq 0,\ u\in E_\ell,\ y,z\in J.
\]
\end{assumption}
Under Assumption~\ref{ass:cond}, for any bounded Borel measurable $g$ and $s\in[0,1]$, the process $M^g$ defined
by
\[
M^g_t:=\bbE[g(Y_s)|\cF^Z_t]=\int_{J}\int_{J}g(y)  \Phi(s,t,Z_t,y,z)m(dz)m(dy),\quad t \geq s,
\]
is a martingale time-indexed by $t$.
We also make the following assumption on $\Phi$:

\begin{assumption}\label{ass:rho}  For all continuous functions $g:J\to\bbR$ with compact support, for each $(s,t)\in [0,1)^2$ with $s<t$, and $dm\otimes dm$-a.a. $(y,z)\in J^2$, we have 
\begin{itemize}
\item[(i)] $\int_{J}\int_{J}g(y)  \Phi(s,\cdot,\cdot,y,z)m(dz)m(dy) \in C^{1,2}([0,1)\times E,\bbR)$,
\item[(ii)] $\Phi(s,t,\cdot ,y,z)\in C^{1}( E,\bbR_{++})$,  and
\item[(iii)] $\forall x \in E\; \exists \eps>0$  such that $\sup_{u \in (x-\eps,x+\eps)}|\Phi_u(s,t,u,\cdot,\cdot)|\in L^1(J^2,dm\otimes dm)$.
\end{itemize}
Moreover, there exists a function
$\Gamma$, independent of $(s,y)$,
such that 
\[
\Gamma(t,u,z)=\frac{\Phi_u(s,t,u,y,z)}{\Phi(s,t,u,y,z)}.
\]
\end{assumption}

In view of this assumption, we obtain  by It\^o's formula 
\[
M_t^g=M_s^g +\int_s^t \int_{J} \int_{J} g(y)\Phi_x(s,r,Z_r,y,z)m(dz)m(dy)\chf_{Z_r\in E}\sigma(r,Z_r)dW_r.
\]

We use this in order to prove the following theorem.
\begin{theorem}\label{theorem:H}
Let $Z$ be the diffusion process defined by \eqref{eq:Zref} with $s=0$.
Under Assumptions~\ref{ass:mk}-\ref{ass:rho}, the $\bbH$-decomposition of $Z$ is given by
\[
Z_t=Z_0+\int_0^t\chf_{Z_r\in E}\sigma(r,Z_r)d\beta_r + \int_0^t\chf_{Z_r\in E}(b(r, Z_r)+\sigma^2(r,Z_r) \Gamma(r,Z_r,Y_r)) dr,
\]
where $\beta$ is an $\bbH$-Brownian motion.
\end{theorem}

\begin{proof}
In order to find the $\bbH$-decomposition of $Z$ we need to find
$\bbE\left[\int_s^t \chf_{Z_r\in E}\sigma(r,Z_r)dW_r | \cH_s\right]$. However, due to the Markov assumption, we only
need to compute the expectations of the following form:
\[
\bbE\left[\int_s^t \chf_{Z_r\in E}\sigma(r,Z_r)dW_r g(Z_s)j(Y_s)\right]
\]
for any continuous  functions $g, j$ with compact support.
Moreover, since $Z$ is continuous and  killed upon exiting $E$, it suffices to consider 
\[
\bbE\left[\int_s^{t\wedge \tau_n} \chf_{Z_r\in E}\sigma(r,Z_r)dW_r g(Z_s)j(Y_s)\right],
\]
where $\tau_n:=\inf\{t\geq s: Z_t \notin ((\ell+\frac{1}{n})\vee -n, n)\}$.

To this end, we have
\bean
&&\bbE\left[\int_s^{t\wedge \tau_n} \sigma(r,Z_r)dW_r g(Z_s)j(Y_s)\right]=\bbE\left[\int_s^{t\wedge \tau_n} \sigma(r,Z_r)dW_r g(Z_s)\bbE[j(Y_s)|\cF^Z_{t\wedge \tau_n}]\right]\\
&=&\bbE\left[\int_s^{t\wedge \tau_n} \sigma(r,Z_r)dW_r g(Z_s)M^g_{t\wedge \tau_n}\right]=
\bbE\left[\int_s^{t\wedge \tau_n}\sigma(r,Z_r)dW_r g(Z_s)(M^g_t-M^g_s)\right]
\\
&=&\bbE\left[\int_s^{t\wedge \tau_n} \sigma(r,Z_r)dW_r g(Z_s)
\int_s^{t\wedge \tau_n} \int_{J^2} j(y)\Phi_x(s,r,Z_r,y,z)m(dz)m(dy)\sigma(r,Z_r)dW_r\right]
\\
&=&\bbE\left[\int_s^{t\wedge \tau_n} \sigma(r,Z_r)dW_r g(Z_s)
\int_s^{t\wedge \tau_n}\int_{J^2} j(y)\Gamma(r,Z_r,z) \Phi(s,r,Z_r,y,z)m(dz)m(dy)\sigma(r,Z_r)dW_r\right]
\\
&=&\bbE\left[g(Z_s)
\int_s^{t\wedge \tau_n}\int_{J^2}j(y)\Gamma(r,Z_r,z) \Phi(s,r,Z_r,y,z)m(dz)m(dy)\sigma^2(r,Z_r)dr \right]
\\
&=&\bbE\left[g(Z_s)\int_s^{t\wedge \tau_n}\sigma^2(r,Z_r)\bbE\left[
j(Y_s)\Gamma(r,Z_r,Y_r)|\cF^Z_r\right]
dr\right]\\
&=&\bbE\left[g(Z_s)j(Y_s)\int_s^{t\wedge \tau_n}
\sigma^2(r,Z_r)  \Gamma(r,Z_r,Y_r)\,dr\right].
\eean
This completes the proof.
\end{proof}

The next example shows how some of the earlier works on
progressive enlargement with random times can be studied using the
method above.
\begin{example} Let $\tau$ be a random time, which is not an
  $\bbF$-stopping time. Define $Y_t:=\tau \wedge t$. Assume $(Z,Y)$ 
  is $\bbH$-Markov.  Note that
\[
\bbP(Y_s \in dy, Y_t \in dz|\cF_t)=\bbP(Y_s \in dy|\cF_t,
Y_t=z)\bbP(Y_t \in dz|\cF_t).
\]
Moreover, given $Y_t$, $Y_s$ becomes non-random for $s \leq t$ so that
we have
\[
\bbP(Y_s \in dy, Y_t \in dz|\cF_t)=\delta_{s \wedge z}(dy) \bbP(Y_t \in
dz|\cF_t).
\]
  Therefore, whether the assumption of Theorem \ref{theorem:H} is satisfied can be checked directly by just looking at $\bbP(Y_t \in
dz|\cF_t)$.

Next, let us take a look at the specific case when $\tau= V(T)$, where $Z$ is a standard Brownian motion,
$T=\inf\{t>0:Z_t=1\}$, and $V$ is a strictly increasing, continuous
and deterministic
function such that $V(t)<t$ for all $t\geq 0$.  Then, for $z \leq t$,
\bean
\chf_{[T>t]}\bbP(Y_t \in dz|\cF_t)&=&\chf_{[T>t]}\bbP(Y_t \in dz, t > V(T)|\cF_t)+\chf_{[T>t]}\bbP(Y_t
\in dz, t \leq V(T)|\cF_t)\\
&=&\chf_{[T>t]}\bbP(T \in dV^{-1}(z)
|\cF_t)+\chf_{[T>t]}\bbP(T \geq V^{-1}(t)|\cF_t) \delta_t(dz).
\eean
Recall that
\[
 \bbP(\tau>t+s|\cF_t)=H(s,Z_t),
\]
where, for $x <1$,
\[
H(s,x)=\int_s^{\infty} \frac{1-x}{\sqrt{2 \pi
     r^3}}e^{-\frac{(1-x)^2}{2 r}}\, dr.
 \]
Then, on the set $[\tau >t]$,
\[
\bbP(Y_t \in dz|\cF_t)=
\chf_{z \in (V(t),t)}\frac{1}{V'(V^{-1}(z))}\frac{1-Z_t}{\sqrt{2 \pi
     (V^{-1}(z)-t)^3}}e^{-\frac{(1-Z_t)^2}{2 (V^{-1}(z)-t)}} dz +
 H(V^{-1}(t)-t,Z_t)\delta_t(dz) .
\]
Hence, the $\bbH$-decomposition of $Z$ is given by
\bean
Z_t=\beta_t+ \int_0^{\tau\wedge t} \frac{H_x(V^{-1}(s)-s,Z_s)}{H(V^{-1}(s)-s,Z_s)}\,
  ds 
-\int_{\tau}^{\tau) \vee t}\left\{ \frac{1}{1-Z_s} -
\frac{1-Z_s}{T -s}\right\}\, ds.
\eean
The second integral implies that after time $\tau$, $1-B$ is a
3-dimensional Bessel bridge. This certainly makes sense. Indeed, since
$V$ is invertible, $T$ becomes known after time $V(T)$ and the
result becomes unsurprising due to the relationship between Bessel bridges and
first hitting times of Brownian motions.
\end{example}
Another application of the formulae above gives the joint decomposition of insider's signal and the equilibrium total demand process when the insider's private signal consists of a standard Brownian motion $Z$ with its end value $Z_1$.
\begin{example} 
Suppose $Z$ is a standard Brownian motion and
  time runs from $0$ to $1$. Let $B$ be another Brownian motion independent of $Z$, and  suppose
 that $Y$ solves with respect to a larger filtration
\[
Y_t=B_t + \int_0^t \frac{Z_1 - Y_s}{1-s}ds.
\]
Recall that  $\bbG$ denote the natural filtration of $Y$ and observe that 
$Y$ is a $\bbG$-Brownian motion (see Lemma 3 in \cite{Back92}). Moreover, $Y_1=Z_1 \, \bbP$-a.s., and $Y$ and $Z$ are independent conditional on $Z_1$. Thus,
\bean
\bbP(Y_s\in dy, Y_t \in dz|\cF_t)&=&\bbE\left[\bbP(Y_s\in dy, Y_t \in
  dz|\cF_t, Y_1)|\cF_t\right] \\
&=&\bbE\left[\bbP(Y_s\in dy, Y_t \in
  dz| Y_1)|\cF_t\right] \\
&=&\int_{\bbR}
\frac{q(s,0,y)q(t-s,y,z)q(1-t,z,w)}{q(1,0,w)}q(1-t,Z_t,w)\,dw dy dz .
\eean
Therefore,
\[
\Phi(s,t,u,y,z)=q(s,0,y)q(t-s,y,z) \int_{\bbR}
\frac{q(1-t,z,w)}{q(1,0,w)}q(1-t,u,w)\,dw
\]
satisfies  Assumption \ref{ass:rho} and
\[
\Gamma(t,u,z)=\frac{\int_{\bbR}
\frac{q(1-t,z,w)}{q(1,0,w)}q_u(1-t,u,w)\,dw}{\int_{\bbR}
\frac{q(1-t,z,w)}{q(1,0,w)}q(1-t,u,w)\,dw }.
\]
Hence, in the filtration jointly generated by $Y$ and $Z$, $Z$ will follow
\[
dZ_t=d\beta^1_t + \Gamma(t,Z_t,Y_t)dt
\]
for some standard Brownian motion $\beta^1$. The dynamics of $Y$ in the joint filtration can be found by symmetry:
\[
dY_t=d\beta^2_t + \Gamma(t,Y_t,Z_t) dt,
\]
where $\beta^2$ is another Brownian motion independent of $\beta^1$.
 \end{example}

\bibliographystyle{siam}
\bibliography{ref}

@article{leonard2012,
  title={From the Schr{\"o}dinger problem to the Monge--Kantorovich problem},
  author={L{\'e}onard, Christian},
  journal={Journal of Functional Analysis},
  volume={262},
  number={4},
  pages={1879--1920},
  year={2012},
  publisher={Elsevier}
}

@article{sch31,
  author  = {Schr{\"o}dinger, Erwin},
  title   = {{\"U}ber die Umkehrung der Naturgesetze},
  journal = {Sitzungsberichte der Preussischen Akademie der Wissenschaften, Physikalisch-mathematische Klasse},
  volume  = {144},
  pages   = {144--153},
  year    = {1931}
}

@inproceedings{sch32,
  title={Sur la th{\'e}orie relativiste de l'{\'e}lectron et l'interpr{\'e}tation de la m{\'e}canique quantique},
  author={Schr{\"o}dinger, Erwin},
  booktitle={Annales de l'institut Henri Poincar{\'e}},
  volume={2},
  number={4},
  pages={269--310},
  year={1932}
}

@article{Marcelnotes,
  title={Introduction to entropic optimal transport},
  author={Nutz, Marcel},
  journal={Lecture notes, Columbia University},
  volume={306},
  number={19},
  pages={307},
  year={2021}
}

@book{lipcer2001statistics,
  title={Statistics of Random Processes II: II. Applications},
  author={Liptser, Robert S and Shiryaev, Albert Nikolaevi{\v{c}}},
  volume={2},
  year={2001},
  publisher={Springer Science \& Business Media}
}

@article{Dai91,
  title={A stochastic control approach to reciprocal diffusion processes},
  author={Dai Pra, Paolo},
  journal={Applied mathematics and Optimization},
  volume={23},
  number={1},
  pages={313--329},
  year={1991},
  publisher={Springer}
}

@incollection{F88,
  title={Random fields and diffusion processes},
  author={F{\"o}llmer, Hans},
  booktitle={{\'E}cole d'{\'E}t{\'e} de Probabilit{\'e}s de Saint-Flour XV--XVII, 1985--87},
  pages={101--203},
  year={1988},
  publisher={Springer}
}

@article{CCdef,
	title={Insider trading in an equilibrium model with default: a passage from reduced-form to structural modelling},
	author={Campi, Luciano and \c{C}etin, Umut},
	journal={Finance and Stochastics},
	volume={11},
	number={4},
	pages={591--602},
	year={2007},
	publisher={Springer}
}

@book{DMB-CD,
	title={Dynamic Markov Bridges and Market Microstructure: Theory and Applications},
	author={{\c{C}}etin, Umut and Danilova, Albina},
	volume={90},
	year={2018},
	publisher={Springer}
}

@article{Kyle,
	title={Continuous auctions and insider trading},
	author={Kyle, Albert S},
	journal={Econometrica: Journal of the Econometric Society},
	pages={1315--1335},
	year={1985},
	publisher={JSTOR}
}

@article{Back92,
	title={Insider trading in continuous time},
	author={Back, Kerry},
	journal={The Review of Financial Studies},
	volume={5},
	number={3},
	pages={387--409},
	year={1992},
	publisher={Oxford University Press}
}

@article{BP98,
	title={Long-lived information and intraday patterns},
	author={Back, Kerry and Pedersen, Hal},
	journal={Journal of financial markets},
	volume={1},
	number={3-4},
	pages={385--402},
	year={1998},
	publisher={Elsevier}
}

@article{CDRA,
	title={Markovian Nash equilibrium in financial markets with asymmetric information and related forward--backward systems},
	author={{\c{C}}etin, Umut and Danilova, Albina},
	journal={The Annals of Applied Probability},
	volume={26},
	number={4},
	pages={1996--2029},
	year={2016},
	publisher={Institute of Mathematical Statistics}
}

@article{CCDdef,
	title={Equilibrium model with default and dynamic insider information},
	author={Campi, Luciano and {\c{C}}etin, Umut and Danilova, Albina},
	journal={Finance and Stochastics},
	volume={17},
	number={3},
	pages={565--585},
	year={2013},
	publisher={Springer}
}

@article{D,
	title={Stock market insider trading in continuous time with imperfect dynamic information},
	author={Danilova, Albina},
	journal={Stochastics An International Journal of Probability and Stochastics Processes},
	volume={82},
	number={1},
	pages={111--131},
	year={2010},
	publisher={Taylor \& Francis}
}

@article{CX13,
	title={Point process bridges and weak convergence of insider trading models},
	author={{\c{C}}etin, Umut and Xing, Hao},
	journal={Electronic Journal of Probability},
	volume={18},
	year={2013},
	publisher={The Institute of Mathematical Statistics and the Bernoulli Society}
}

@article{CD-GKB,
	title={On Pricing Rules and Optimal Strategies in General Kyle--Back Models},
  author={{\c{C}}etin, Umut and Danilova, Albina},
  journal={SIAM Journal on Control and Optimization},
  volume={59},
  number={5},
  pages={3973--3998},
  year={2021},
  publisher={SIAM}
}

@article{choRA,
	title={Continuous auctions and insider trading: uniqueness and risk aversion},
	author={Cho, Kyung-Ha},
	journal={Finance and Stochastics},
	volume={7},
	number={1},
	pages={47--71},
	year={2003},
	publisher={Springer}
}

@article{BERA,
	title={Kyle-{B}ack Models with risk aversion and non-{G}aussian Beliefs},
	author={Bose, Shreya and Ekren, Ibrahim},  journal={The Annals of Applied Probability},
  volume={33},
  number={6A},
  pages={4238--4271},
  year={2023},
  publisher={Institute of Mathematical Statistics}
}

@article{CetRH,
	title={Financial equilibrium with asymmetric information and random horizon},
	author={{\c{C}}etin, Umut},
	journal={Finance and Stochastics},
	volume={22},
	number={1},
	pages={97--126},
	year={2018},
	publisher={Springer}
}

@article{CDbridge,
  title={Markov bridges: {SDE} representation},
  author={{\c{C}}etin, Umut and Danilova, Albina},
  journal={Stochastic Processes and their Applications},
  volume={126},
  number={3},
  pages={651--679},
  year={2016},
  publisher={Elsevier}
}

@book{Pinsky,
  title={Positive harmonic functions and diffusion},
  author={Pinsky, Ross G},
  volume={45},
  year={1995},
  publisher={Cambridge university press}
}

@article{KpenC,
title = {Insider trading with penalties in continuous time},
journal = {Journal of Economic Theory},
volume = {228},
pages = {106061},
year = {2025},
issn = {0022-0531},
doi = {https://doi.org/10.1016/j.jet.2025.106061},
url = {https://www.sciencedirect.com/science/article/pii/S0022053125001073},
author = {\c{C}etin, Umut},
}

@article{BackEkren,
  title={Optimal Transport and Informed Trading in Securities Markets},
  author={Back, Kerry and Cocquemas, Fran{\c{c}}ois and Ekren, Ibrahim and Lioui, Abraham},
  journal={Available at SSRN 3628726},
  year={2024}
}

@article{EMZsv,
  title={Kyle’s model with stochastic liquidity},
  author={Ekren, Ibrahim and Mostowski, Brad and {\v{Z}}itkovi{\'c}, Gordan},
  journal={Finance and Stochastics},
  volume={29},
  number={4},
  pages={1195--1231},
  year={2025},
  publisher={Springer}
}

@article{BEmultdRA,
  title={Multidimensional Kyle--Back model with a risk averse informed trader},
  author={Bose, Shreya and Ekren, Ibrahim},
  journal={SIAM Journal on Financial Mathematics},
  volume={15},
  number={1},
  pages={93--120},
  year={2024},
  publisher={SIAM}
}

@article{choilarsenkyle,
  title={Trading constraints in continuous-time Kyle models},
  author={Choi, Jin Hyuk and Kwon, Heeyoung and Larsen, Kasper},
  journal={SIAM Journal on Control and Optimization},
  volume={61},
  number={3},
  pages={1494--1512},
  year={2023},
  publisher={SIAM}
}

@article{CSLtarget,
  title={Information and trading targets in a dynamic market equilibrium},
  author={Choi, Jin Hyuk and Larsen, Kasper and Seppi, Duane J},
  journal={Journal of Financial Economics},
  volume={132},
  number={3},
  pages={22--49},
  year={2019},
  publisher={Elsevier}
}

@article{QZnewkyle,
  title={A New Approach for the Continuous Time Kyle-Back Strategic Insider Equilibrium Problem},
  author={Qiao, Bixing and Zhang, Jianfeng},
  journal={arXiv preprint arXiv:2506.12281},
  year={2025}
}

@article{ma2018kyle,
  title={Kyle--Back equilibrium models and linear conditional mean-field SDEs},
  author={Ma, Jin and Sun, Rentao and Zhou, Yonghui},
  journal={SIAM Journal on Control and Optimization},
  volume={56},
  number={2},
  pages={1154--1180},
  year={2018},
  publisher={SIAM}
}

@book{ladyzhenskaia,
  title={Linear and quasi-linear equations of parabolic type},
  author={Ladyzhenskaia, Olga Aleksandrovna and Solonnikov, Vsevolod Alekseevich and Ural'tseva, Nina N},
  volume={23},
  year={1968},
  publisher={American Mathematical Soc.}
}
\end{document}